\documentclass[letterpaper,article,oneside,11pt]{memoir}

\usepackage[dvipsnames]{xcolor}
\usepackage{amsmath}
\usepackage{amssymb}
\usepackage{amsthm}
\usepackage{mathrsfs}
\usepackage{lmodern}
\usepackage[T1]{fontenc}
\usepackage[utf8]{inputenc}
\usepackage{enumitem}
\usepackage{tikz}
\usetikzlibrary{decorations.pathreplacing,calligraphy}
\usepackage[capitalize]{cleveref}
\usepackage[backend=biber]{biblatex}
\settrims{0pt}{0pt}
\setheadfoot{0.5in}{0.5in}
\setulmarginsandblock{1.0in}{1.0in}{*}
\setlrmarginsandblock{1.0in}{1.0in}{*}
\checkandfixthelayout

\counterwithout{equation}{section}
\counterwithout{section}{chapter}

\renewcommand{\maketitle}{
 {\centering\textbf{\thetitle}\par
  \vspace{0.5em}
  \centering{\theauthor}\par 
  \vspace{2em}
 }
}

\setsecheadstyle{\bfseries}

\setsecnumdepth{subsubsection}
\setsecnumformat{\csname the#1\endcsname. }
\setsubsecheadstyle{\normalfont\bfseries}
\setaftersubsecskip{-0.5em}
\setsubsubsecheadstyle{\normalfont\bfseries}
\setaftersubsubsecskip{-0.5em}

\OnehalfSpacing
\newtheorem{theorem}{Theorem}

\newtheorem{proposition}[theorem]{Proposition}
\newtheorem{corollary}[theorem]{Corollary}
\newtheorem{lemma}[theorem]{Lemma}

\theoremstyle{definition}
\newtheorem{example}[theorem]{Example}
\newtheorem{question}[theorem]{Question}
\newtheorem{remark}[theorem]{Remark}

\newtheorem{definition}[theorem]{Definition}

\newcommand{\metric}{\mathsf{d}}

\newcommand{\intd}{\;\mathsf{d}}

\newcommand{\N}{\mathbb{N}}
\newcommand{\Z}{\mathbb{Z}}

\newcommand{\meas}{\mathcal{M}}

\DeclareMathOperator{\symdiff}{\triangle}
\newcommand{\lp}{\mathsf{L}}
\newcommand{\ltwo}{\lp^{\!\mathsf{2}}}

\newcommand{\KR}{\mathsf{KR}}

\newcommand{\id}{\mathsf{Id}}

\newcommand{\lft}{\mathsf{\mathsf{L}}}
\newcommand{\rht}{\mathsf{\mathsf{R}}}
\DeclareMathOperator{\lst}{\mathsf{e}}
\DeclareMathOperator{\bad}{\mathsf{V}}

\DeclareMathOperator{\p}{\mathsf{p}}
\DeclareMathOperator{\q}{\mathsf{q}}
\renewcommand{\P}{\mathsf{P}}
\DeclareMathOperator{\Q}{\mathsf{Q}}

\newcommand{\scl}{\mathsf{J}}

\newcommand{\define}[1]{\textbf{#1}}
\renewcommand{\emptyset}{\varnothing}

\definecolor{dg}{RGB}{0,128,0}

\usepackage[normalem]{ulem}

\begin{document}

\title{A rank one mild mixing system without minimal self joinings}
\author{Jon Chaika \and Donald Robertson}
\maketitle

\begin{abstract}
We show that there is a rank 1 transformation that is mildly mixing but does not have minimal self-joinings, answering a question of Thouvenot.
\end{abstract}

\section{Introduction}

A \define{joining} of two measure-preserving system $(X,\mu,T)$ and $(Y, \nu,S)$ is any $T \times S$ invariant measure on $X \times Y$ that is mapped to $\mu$ by the first coordinate projection and to $\nu$ by the second.
The product measure $\mu \otimes \nu$ is always a joining.
The existence of other joinings of $(X\,\mu,T)$ and $(Y,\nu,S)$ detects whether the systems have properties in common.
In particular, if  $(X,\mu,T)$ and $(Y,\nu,S)$ have no other joinings they tend to have orthogonal dynamical properties.
When $\mu \otimes \nu$ is the only joining the systems $(X,\mu,T)$ and $(Y,\nu,S)$ are called \define{disjoint}.
Joinings and disjointness were first studied by Furstenberg~\cite{MR0213508}.
The survey of de la Rue~\cite{MR4647076} is a good introduction.

A non-trivial system $(X,\mu,T)$ is never disjoint from itself: there are always \define{off-diagonal} joinings obtained by embedding $X$ in $X \times X$ using the map $x \mapsto (x,T^n x)$ for any $n \in \Z$.
When these and the product measure are the only ergodic joinings of $(X,\mu,T)$ with itself the system $(X,\mu,T)$ is said to have \define{minimal self-joinings}.
Rudolph~\cite{MR0555301} gave the first example of a system with minimal self-joinings.

King~\cite{MR0963154} showed that all mixing rank one systems have minimal self-joinings.
The survey of Ferenczi~\cite{MR1436950} collects many equivalent definitions of rank one and we record here the definition most useful for our purposes.

\begin{definition}[{\cite[Definition~6]{MR1436950}}]
A system $(X,\mu,T)$ is \define{rank one} if for every measurable set $A \subset X$ and every $\epsilon > 0$ there is a measurable set $F \subset X$ and $h \in \N$ and a measurable set $A' \subset X$ such that
\begin{itemize}
\item 
the sets $F, TF,\dots,T^{h-1} F$ are pairwise disjoint
\item
$\mu(A' \symdiff A) < \epsilon$
\item
$\mu(F \cup TF \cup \cdots \cup T^{h-1} F) > 1 - \epsilon$
\item
$A'$ belongs to the $\sigma$-algebra generated by $\{ F, TF,\dots, T^{h-1} F\}$
\end{itemize}
all hold.
\end{definition}

Thouvenot asked the following question.

\begin{question}[{\cite[1.2 in Section~12]{MR4232241}}]
\label{q:thouvenot}
Does a non-rigid rank one transformation without factors have the minimal self-joinings property?
\end{question}

To make sense of this question we recall the following definitions.

\begin{definition}
Fix a system $(X,\mu,T)$.
Say that $f \in \ltwo(X,\mu)$ is \define{rigid} for $T$ if there is a sequence $\textit{\textbf{r}}$ in $\Z$ with $|r(n)| \to \infty$ and $T^{r(n)} f \to f$ in $\ltwo(X,\mu)$.
In this situation we say that $f$ is \define{rigid} along $\textbf{\textit{r}}$.
Any such sequence is called an $f$-\define{rigidity sequence}.
One says that $(X,\mu,T)$ is \define{mild mixing} if the only rigid functions are the constant functions.
\end{definition}

King~\cite[Page~377]{MR0863200} proved that all non-trivial factors of rank one systems are rigid.
As mild mixing systems cannot have non-trivial rigid factors, \cref{q:thouvenot} is equivalent to asking whether mildly mixing rank-one systems have minimal self-joinings.
We construct a mild-mixing rank-one system that does not have minimal self-joinings, thereby answering negatively \cref{q:thouvenot}.

\begin{theorem}
\label{thm:main}
There is a mild mixing rank one system $(X,\mu,T)$ without minimal self-joinings.
\end{theorem}

Ryzhikov~\cite{Ryzarx} has results connected to this theorem. On the one hand he showed that such an example can not be ``bounded'' rank 1. That is, bounded rank 1 mildly mixing systems have minimal self-joinings \cite{MR3235793}. 
On the other hand, he showed that for any $\epsilon>0$ there is a mildly mixing transformation without minimial self-joinings that has a ``local rank'' of $1-\epsilon$.

\subsection{Outline of proof}

We define in \cref{sec:system} a measure-preserving system $(X,\mu,T)$.
The following properties of the system are established in Sections \ref{sec:rank one}, \ref{sec:ergodic joining} and \ref{sec:mild mixing} respectively.

\begin{theorem}
\label{thm:rank one}
The system constructed in \cref{sec:system} is rank one.
\end{theorem}

\begin{theorem}
\label{thm:ergodic joining}
The system constructed in \cref{sec:system} does not have the minimal self-joinings property.
\end{theorem}

\begin{theorem}
\label{thm:mild mixing}
The system constructed in \cref{sec:system} is mild mixing.
\end{theorem}

Our system will be the the induced transformation of an odometer on an explicit subset obtained by removing a countable number cylinders.
First consider the odometer
\[
\Omega = \prod_{i \in \N} \{0, \dots, 9 \}
\]
on which $S : \Omega \to \Omega$ is the usual ``addition with carry''.
The product of the uniform measures on $\{0,\dots,9\}$, denoted $\nu$, is $S$-invariant and the system $(\Omega, \nu, S)$ is ergodic.
Put
\[
W_i = \{ x \in \Omega : x(1) = 7,\dots, x(i-1) = 7, x(i) = 8 \}
\]
for all $i \ge 2$ and $W_1 = \{ x \in \Omega : x(1) = 8 \}$.
The closed subset
\[
X = \Omega \bigg\backslash \bigcup_{i \in \N} W_i
\]
has positive measure and we can therefore consider the induced system $(X,\mu,T)$ where $\mu$ is the normalization of the restriction of $\nu$ to $X$ and $T$ is the first return map defined by
\[
T(x) = S^{\min \{ n \in \N : S^n(x) \in X\}}(x)
\]
for (in this case) all $x \in X$.

The removal of the cylinder sets $W_i$ dramatically alters the dynamical properties of the system.
While $(\Omega,\nu,S)$ is rigid, the system $(X,\mu,T)$ is mild mixing.
This is due to the abundance for any large $n \in \N$ of pairs of nearby points $x,y$ with the property that $T^n(x)$ and $T^{n+1}(y)$ are also nearby.
For example, if $n = 10^i$ and $x(i) = 8$ and $y(i) = 9$ and $x(j) = y(j)$ for all $j < i$ and $(x(1),\dots,x(i-1))$ precedes $(7,\dots,7)$ lexicographically in the collection of cylinders that intersect $X$ then
\begin{itemize}
\item
$x$ and $y$ are close
\item
there will be $1 \le k \le n$ with $S^k(x) \in W_i$
\item
for no $1 \le k \le n$ will $S^k(y)$ belong to any $W_j$
\end{itemize}
so $T^n(x)$ will be close to $T^{n+1}(y)$.
One can use such points to show that every rigid function is approximately $T$-invariant and therefore constant by ergodicity.
For more general $n$ the decimal expansion of $n$ dictates the scale $i$ for which the removed cylinder $W_i$ can be used to ``pick up invariance'' in the above sense. 

It is likely, using an argument analogous to that of \cite{MR2129107} for showing that linear recurrent 3-IETs have minimal self-joinings, that pairs of points with the above behaviour can be used to prove the system $(X,\mu,T)$ has the much stronger property of minimal self-joinings.

To produce a system that does not have minimal self-joinings we build on a method from \cite{MR4269425} for producing ergodic joinings that are neither one of the off-diagonal joinings nor the product measure.
To do so we modify $\Omega$ to be of the form
\[
\Omega = \prod_{i \in \N} \{ 0,\dots,c(i) \}
\]
where $c(i) = 9$ unless $i = 2^{j+5}$ in which case $c(i) = 2^{j+6} - 1$; and by modifying $W_i$ when $i = 2^{j+5}$ to consist of many more cylinders.
The formal definition is in \cref{sec:system}.

In modifying the odometer to allow the removal of many cylinders at certain scales we are able to find for each $k \in \N$ sets $A_k$ and $B_k$  of measure about 0.5 and times $\p(k)$, $\q(k)$ with
\begin{equation}
\label{intro:reduction}
T^{\p(k)} \approx I \textup{ on } A_k
\textup{ and }
T^{\p(k)} \approx T^{\q(k)} \textup{ on } B_k
\end{equation}
and the relation $q_{k+1} = 2p_k - q_k$.
We use these properties to build a non-trivial joining in \cref{sec:ergodic joining}.
See in particular the paragraph after \cref{prop:ce}, for how these properties facilitate building the joining.

Having modified the sets $W_i$ it is no longer clear that the resulting system $(X,\mu,T)$ is mild mixing; the bulk of our effort is in verifying that property.
We wish to pursue the same strategy as described above, but the decimal expansion of a time $n$ may no longer be relevant.
Instead one must expand
\begin{equation}
\label{intro:expansion}
n = e(\scl(n)) \p(\scl(n)) + e(\scl(n) - 1) \p(\scl(n) - 1) + \cdots + e(1) \p(1)
\end{equation}
where $k \mapsto \p(k)$ is a specific sequence of quickly growing return times for our system $(X,\mu,T)$.

When the scale $\scl(n)$ of the expansion \eqref{intro:expansion} is not of the form $2^{j+5}$ we can again ``pick up invariance'' using the single removed cylinder $W_{\scl(n)}$.
However, if $\scl(n)$ is of the form $2^{j+5}$ we cannot take this approach: the set $W_{\scl(n)}$ that we remove consists in this case of too many cylinders to run that argument.
To overcome this difficulty we appeal to the relations \eqref{intro:reduction} which will enable us to ``reduce'' our attention to $n - e(\scl(n)) \p(\scl(n))$ and $n - (\p(\scl(n)) - \q(\scl(n))) e(\scl(n))$ on $A_{\scl(n)}$ and $B_{\scl(n)}$ respectively.
A single application of this reduction procedure may not suffice, as these reduced times may themselves not have expansions at scales where one can pick up invariance.
Moreover, even after several iterations of this reduction procedure we might not be in a position to pick up invariance.
To overcome this difficulty we consider also the reductions of a multiple $sn$ of $n$.
Our main technical result -- \cref{prop:reduction} -- keeps track of what happens as we apply the reduction procedure to $n$ and $sn$ simultaneously. The upshot is that the reduction of $sn$ is roughly equal to $s$ multiplied by the reduction of $n$. Fixing in advance a relatively small scale compared to that of $n$ we choose $s$ large enough that $s$ multiplied by the reduction of $n$ must ``overflow'' and have an expansion at which one can again pick up invariance.

This approach is quite technical, and it is natural to wonder whether the careful bookeeping is necessary.
Two facts suggest this care is important.
The first fact is the  real possibility of ``secret'' rigidity times.
Such ``secret'' times are present in exchanges of three intervals -- see for instance Case~2 in the proof of \cite[Theorem~5.1]{MR2129107} -- and transformations such as Katok's map -- see \cite[Section~1.4.3]{MR1436950}.
The second fact is that there are some natural candidates for rigidity sequences for our transformation: we build our joining by showing that there is a sequence $n_1,n_2,...$ in $\mathbb{Z}$ so that the corresponding sequence of off-diagonal joinings is Cauchy in the weak$^*$ topology; it is then natural to wonder if say, ${n_{i+1}-n_i}$ is a rigidity sequence. 
It is an open question of King~\cite[Page~382]{MR0863200} whether rank one transformations always have that $\{(id \times T^n)_*\mu\}_{n \in \mathbb{Z}}$ 
is weakly dense in the set of ergodic self-joinings.

We hope that our strategy will help to further differentiate mildly mixing rank one transformations from mixing rank one transformations. We comment further on this in \cref{rem:rank two}.

\subsection{Outline of paper}
In \cref{sec:system} we define our system $(X,\mu,T)$ and introduce the numeration system relevant to it. Additionally, we prove some basic properties of the system including that it is partially rigid and weakly mixing.
We prove in \cref{sec:rank one} and \cref{sec:ergodic joining} that our system is rank one and has a non-trivial ergodic joining, respectively. 
The most involved argument -- that $(X,\mu,T)$ is mild mixing -- is in \cref{sec:mild mixing}. We include two appendices: one proves a minor modification of work in \cite{MR4269425} needed in \cref{sec:ergodic joining} and the other lists notation used throughout the paper.

\subsection{Acknowledgments:} We thank Mariusz Lemanczyk for bringing this question to our attention and a helpful conversation. JC was partially supported by NSF grants DMS-2055354 and DMS-2350393.

\section{Building the system}
\label{sec:system}

In this section we define our measure-preserving system $(X,\mu,T)$ from an odometer $(\Omega,\nu,S)$ by taking $X$ to be a subset of $\Omega$ and $T$ to be the induced transformation and $\mu$ to be the normalized restriction of $\nu$ to $X$.
In fact $X$ will be $\Omega$ less a countable union of odometer cylinders and thus a closed subset of $\Omega$.

\subsection{An odometer}

Define
\begin{align*}
d(j) &{} = 2^{j+5} \\
a(j) &{} = 2^{j+5}
\end{align*}
for all $n \in \N$.
Define
\[
c(n)
=
\begin{cases}
9 & n \notin \{ a(j) : j \in \N \} \\
2 d(j) + 1 & n = a(j)
\end{cases}
\]
for all $n \in \N$ and put
\[
\Omega = \prod_{n \in \N} \{0,\dots,c(n) \}
\]
with the product topology.

Given $x,y \in \Omega$ and $k \in \N$ write
\[
x \stackrel{k}{=} y
\]
when $x(i) = y(i)$ for all $1 \le i \le k$.

For concreteness, define a metric on $\Omega$ by 
\[
\metric(x,y)
=
\begin{cases}
0 & \text{ if }x=y \\
2^{-j} & \text{ where } j = \min\{i:x(i)\neq y(i)\}
\end{cases}
\]
for all $x,y \in \Omega$.
Define $S : \Omega \to \Omega$ as follows: put $S(c) = 0$ and, for every $x \ne c$ define
\[
S(x(1),x(2),\dots) = (0,\dots,0,x(j) + 1,x(j+1),\dots)
\]
where $j = \min \{ k \in \N : x(k) < c(k) \}$.
The map $S$ is a homeomorphism of $\Omega$ and preserves the product Borel measure
\[
\nu = \bigotimes_{k \in \N} \frac{1}{c(k) + 1} ( \delta_0 + \cdots + \delta_{c(k)} )
\]
on $\Omega$.

By a \define{cylinder} we mean any subset of $\Omega$ of the form
\[
\{ x \in \Omega : x(i(1)) = b(1),\dots,x(i(k)) = b(k) \}
\]
for some $k \in \N$ and values $0 \le b(j) \le c(i(j))$ for each $1 \le j \le k$. We call $i(1),\dots,i(k)$ the \define{defining indices} of the cylinder.
The \define{length} of a cylinder is the numbers of coordinates it fixes.
We sometimes write $|C|$ for the length of the cylinder $C$.
We use the notation
\begin{equation}
\label{eqn:cylinder}
[b(1) b(2) \cdots b(k)]
=
\{ x \in \Omega : x(1) = b(1) \bmod c(1),\dots,x(k) = b(k) \bmod c(k) \}
\end{equation}
given prescribed values $b(1), \dots, b(k) \in \N \cup \{0\}$ for cylinders defined by the first $k$ coordinates.
We for example write $[0^k]$ for the cylinder $[00 \cdots 0]$ of length $k$ and $[7^{k-1} 8]$ for the cylinder $[77 \cdots 78]$ of length $k$.

We sometimes refer to the index $n \in \N$ defining an odometer digit as a \define{scale}. When $n$ belongs to $\{ a(j) : j \in \N \}$ we call $n$ an \define{unbounded} scale; otherwise $n$ is a \define{bounded} scale.
For any index $n \ge 2$ we picture the odomoter as a collection of $t(n)$ cylinders arranged as in Figure~\ref{fig:odometerTower} mapped to one another by the transformation $S$.

For each $s \in \N$ define $\lambda(s)$ to be the largest unbounded scale strictly smaller than $s$.
Thus $\lambda(a(j)+1) = a(j)$ and $\lambda(a(j)) = a(j-1)$.
For example $\lambda(2^{j+5}) = 2^{j+4}$.
We will sometimes iterate $\lambda$ so that $\lambda^2(a(j)+1) = a(j-1)$.
For example $\lambda^3(2^{j+5} + 6) = 2^{j+3}$.

Put
\[
t(n) = \prod_{j=1}^n 1 + c(j)
\]
for each $n \in \N$.
Note that $t(n)$ is the minimal $m \in \N$ with $S^m[0^{n+1}]=[0^n 1]$.

\begin{figure}[t]
\centering
\begin{tikzpicture}[yscale=0.5, line width=0.5mm]

\begin{scope}[shift={(-2,0)}]
\node at (0,4) {$c(1) c(2) \cdots c(n-1)$};
\node at (0,3.2) {$\vdots$};
\node at (0,2) {$20 \cdots 0$};
\node at (0,1) {$10 \cdots 0$};
\node at (0,0) {$00 \cdots 0$};
\end{scope}

\begin{scope}
\draw (0,4) -- (1,4);
\node at (0.5,3.2) {$\vdots$};
\draw (0,2) -- (1,2);
\draw (0,1) -- (1,1);
\draw (0,0) -- (1,0);
\node at (0.5,-1) {$0$};
\end{scope}

\begin{scope}[shift={(1.5,0)}]
\draw (0,4) -- (1,4);
\node at (0.5,3.2) {$\vdots$};
\draw (0,2) -- (1,2);
\draw (0,1) -- (1,1);
\draw (0,0) -- (1,0);
\node at (0.5,-1) {$1$};
\end{scope}

\begin{scope}[shift={(3,0)}]
\node at (0.5,3) {$\cdots$};
\end{scope}

\begin{scope}[shift={(4.5,0)}]
\draw (0,4) -- (1,4);
\node at (0.5,3.2) {$\vdots$};
\draw (0,2) -- (1,2);
\draw (0,1) -- (1,1);
\draw (0,0) -- (1,0);
\node at (0.5,-1) {$c(n)$};
\end{scope}
\end{tikzpicture}

\caption{A partition of an $\Omega$ into cylinders (drawn schematically as intervals) of length $n$. Rows are labeled by the first $n - 1$ coordinates and columns are labeled by the $n$th coordinate. The transformation $S$ maps a cylinder to the one above, with the cylinder at the top of each column mapped to the bottom cylinder of the tower to the right, and the top right cylinder mapped to the bottom left cylinder.}
\label{fig:odometerTower}
\end{figure}
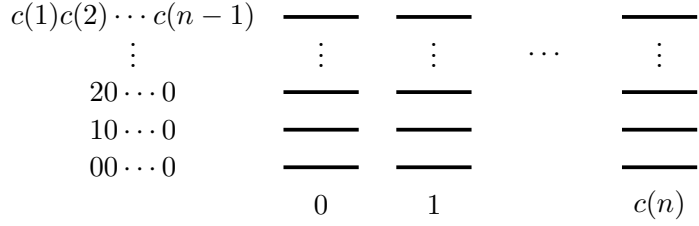

\subsection{Inductive removals to define the system}
\label{subsec:removals}

For each $k \in \N$ we define a union $W_k$ of cylinders of length $k$ that will be removed from $\Omega$ to form our space $X$.
When $k \notin \{ a(n) : n \in \N \}$ the set $W_k$ will consist of the single cylinder $[7^{k-1}8]$ of length $k$.
When $k$ is equal to some $a(n)$ the set $W_k$ will consist of a large number of cylinders of length $k$ defined by an inductive procedure.
The formal construction follows.

Put $Y_0 = \Omega$.
Put $W_1 = [8]$ and for each $1 < k < a(1)$ define
\[
W_k
=
\{ x \in \Omega : x(i) = 7 \textup{ for all } 1 \le i < k \textup{ and } x(k) = 8 \}
=
[7^{k-1}8]
\]
so that $\nu(W_k) = 1/t(k)$ for all $1 \le k < a(1)$.
Write
\[
Y_k = \Omega \setminus ( W_1 \cup \cdots \cup W_k )
\]
and let $S|Y_k$ be the induced transformation on $Y_k$ for all $1 \le k < a(1)$.
See Figure~\ref{fig:boundedRemoval} for a schematic.
For each $1 \le k < a(1)$ define $\q(k) = 0$ and $\p(k)$ to be the minimal $m \in \N$ with $(S|Y_k)^m [0^{k-1}0] = [0^{k-1}1]$.
For example $\p(1) = 1$ and $\p(2) = 9$ and $\p(3) = 89$.

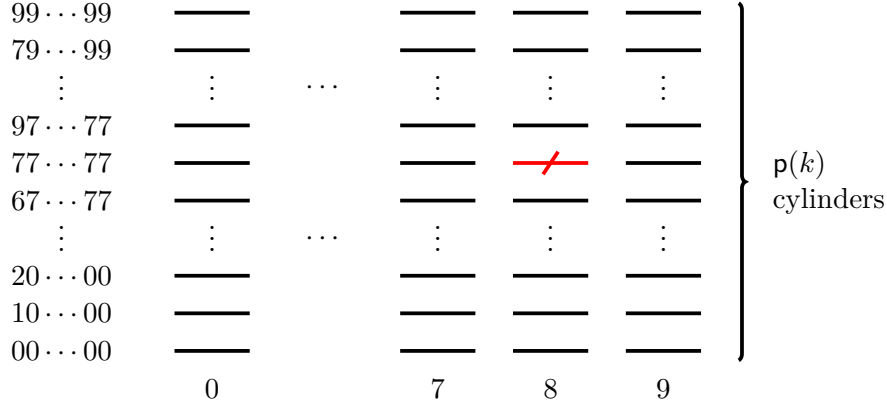
\begin{figure}[t]
\centering
\begin{tikzpicture}[yscale=0.5, line width=0.5mm]

\begin{scope}[shift={(-1.5,0)}]
\node at (0,9) {$99 \cdots 99$};
\node at (0,8) {$79 \cdots 99$};
\node at (0,7.2) {$\vdots$};
\node at (0,6) {$97 \cdots 77$};
\node at (0,5) {$77 \cdots 77$};
\node at (0,4) {$67 \cdots 77$};
\node at (0,3.2) {$\vdots$};
\node at (0,2) {$20 \cdots 00$};
\node at (0,1) {$10 \cdots 00$};
\node at (0,0) {$00 \cdots 00$};
\end{scope}

\begin{scope}
\draw (0,9) -- (1,9);
\draw (0,8) -- (1,8);
\node at (0.5,7.2) {$\vdots$};
\draw (0,6) -- (1,6);
\draw (0,5) -- (1,5);
\draw (0,4) -- (1,4);
\node at (0.5,3.2) {$\vdots$};
\draw (0,2) -- (1,2);
\draw (0,1) -- (1,1);
\draw (0,0) -- (1,0);
\node at (0.5,-1) {$0$};
\end{scope}

\begin{scope}[shift={(1.5,0)}]
\node at (0.5,7) {$\cdots$};
\node at (0.5,3) {$\cdots$};
\end{scope}

\begin{scope}[shift={(3,0)}]
\draw (0,9) -- (1,9);
\draw (0,8) -- (1,8);
\node at (0.5,7.2) {$\vdots$};
\draw (0,6) -- (1,6);
\draw (0,5) -- (1,5);
\draw (0,4) -- (1,4);
\node at (0.5,3.2) {$\vdots$};
\draw (0,2) -- (1,2);
\draw (0,1) -- (1,1);
\draw (0,0) -- (1,0);
\node at (0.5,-1) {$7$};
\end{scope}

\begin{scope}[shift={(4.5,0)}]
\draw (0,9) -- (1,9);
\draw (0,8) -- (1,8);
\node at (0.5,7.2) {$\vdots$};
\draw (0,6) -- (1,6);
\draw[color=red] (0,5) -- (1,5);
\draw[color=red] (0.4,4.7) -- (0.6,5.3);
\draw (0,4) -- (1,4);
\node at (0.5,3.2) {$\vdots$};
\draw (0,2) -- (1,2);
\draw (0,1) -- (1,1);
\draw (0,0) -- (1,0);
\node at (0.5,-1) {$8$};
\end{scope}

\begin{scope}[shift={(6,0)}]
\draw (0,9) -- (1,9);
\draw (0,8) -- (1,8);
\node at (0.5,7.2) {$\vdots$};
\draw (0,6) -- (1,6);
\draw (0,5) -- (1,5);
\draw (0,4) -- (1,4);
\node at (0.5,3.2) {$\vdots$};
\draw (0,2) -- (1,2);
\draw (0,1) -- (1,1);
\draw (0,0) -- (1,0);
\node at (0.5,-1) {$9$};
\end{scope}

\begin{scope}[shift={(7.5,0)}]
\draw [decorate, decoration = {brace}] (0,9.25) --  (0,-0.25);
\node[align=left] at (1.2,4.5) {$\p(k)$\\cylinders};
\end{scope}

\end{tikzpicture}

\caption{The picture of $Y_{k-1}$ at scale $k$ for some $1 < k < a(1)$. The marked cylinder is $W_k$. Note that the rows labeled $89 \cdots 99$ and $87 \cdots 77$ are not shown as they were removed via $W_1$.}
\label{fig:boundedRemoval}

\end{figure}

Put $\Q(1) = \q(a(1)) = 1$ and define
\[
W_{a(1)}
=
\bigcup_{j=1}^{\Q(1)}
\{ x \in Y_{a(1)-1} : x(a(1)) \ge d(1) + 1 \textup{ and } (S|Y_{a(1)-1})^j(x) \in [0^{a(1) -1}]\} = \bigcup_{i=d(1)+1}^{2d(1)+1} [9^{a(1)-1} i]
\] 
cf.\ Figure~\ref{fig:firstBigRemoval}.
We have $\nu(W_{a(1)}) = (d(1) + 1) \Q(1) /t(a(1))$.
Put
\[
Y_{a(1)} = \Omega \setminus ( W_1 \cup \cdots \cup W_{a(1)} )
\]
and let $S|Y_{a(1)}$ be the transformation induced on $Y_{a(1)}$ by $S$.
Analogous to what we did before, define $\p(a(n))$ to be the minimal $m \in \N$ with $(S|Y_{a(1)})^m [0^{a(1)-1}0] = [0^{a(1) - 1} 1]$.
Write $\P(1) = \p(a(n))$.

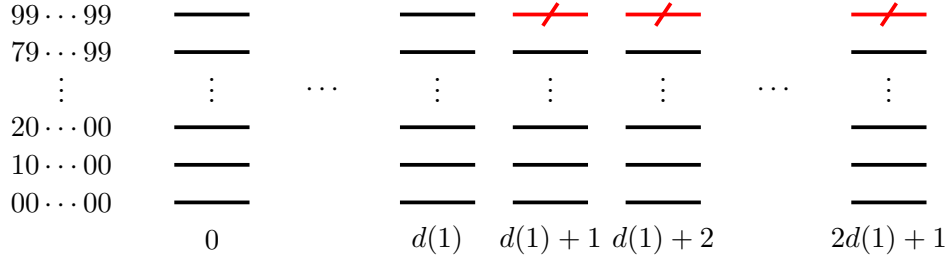
\begin{figure}[t]
\centering
\begin{tikzpicture}[yscale=0.5, line width=0.5mm]

\begin{scope}[shift={(-1.5,0)}]
\node at (0,5) {$99 \cdots 99$};
\node at (0,4) {$79 \cdots 99$};
\node at (0,3.2) {$\vdots$};
\node at (0,2) {$20 \cdots 00$};
\node at (0,1) {$10 \cdots 00$};
\node at (0,0) {$00 \cdots 00$};
\end{scope}

\begin{scope}
\draw (0,5) -- (1,5);
\draw (0,4) -- (1,4);
\node at (0.5,3.2) {$\vdots$};
\draw (0,2) -- (1,2);
\draw (0,1) -- (1,1);
\draw (0,0) -- (1,0);
\node at (0.5,-1) {$0$};
\end{scope}

\begin{scope}[shift={(1.5,0)}]
\node at (0.5,3) {$\cdots$};
\end{scope}

\begin{scope}[shift={(3,0)}]
\draw (0,5) -- (1,5);
\draw (0,4) -- (1,4);
\node at (0.5,3.2) {$\vdots$};
\draw (0,2) -- (1,2);
\draw (0,1) -- (1,1);
\draw (0,0) -- (1,0);
\node at (0.5,-1) {$d(1)$};
\end{scope}

\begin{scope}[shift={(4.5,0)}]
\draw[color=red] (0,5) -- (1,5);
\draw[color=red] (0.4,4.7) -- (0.6,5.3);
\draw (0,4) -- (1,4);
\node at (0.5,3.2) {$\vdots$};
\draw (0,2) -- (1,2);
\draw (0,1) -- (1,1);
\draw (0,0) -- (1,0);
\node at (0.5,-1) {$d(1)+1$};
\end{scope}

\begin{scope}[shift={(6,0)}]
\draw[color=red] (0,5) -- (1,5);
\draw[color=red] (0.4,4.7) -- (0.6,5.3);
\draw (0,4) -- (1,4);
\node at (0.5,3.2) {$\vdots$};
\draw (0,2) -- (1,2);
\draw (0,1) -- (1,1);
\draw (0,0) -- (1,0);
\node at (0.5,-1) {$d(1)+2$};
\end{scope}

\begin{scope}[shift={(7.5,0)}]
\node at (0.5,3) {$\cdots$};
\end{scope}

\begin{scope}[shift={(9,0)}]
\draw[color=red] (0,5) -- (1,5);
\draw[color=red] (0.4,4.7) -- (0.6,5.3);
\draw (0,4) -- (1,4);
\node at (0.5,3.2) {$\vdots$};
\draw (0,2) -- (1,2);
\draw (0,1) -- (1,1);
\draw (0,0) -- (1,0);
\node at (0.5,-1) {$2d(1)+1$};
\end{scope}

\end{tikzpicture}

\caption{The picture of $Y_{a(1)-1}$ at scale $a(1)$. The marked cylinders constitute $W_{a(1)}$. Note that the row labeled $89 \cdots 99$ is not shown as it was removed via $W_1$.}
\label{fig:firstBigRemoval}

\end{figure}

We now proceed inductively, supposing that subsets $W_1,\dots,W_{a(n)}$ of $\Omega$ and $\p(1),\dots,\p(a(n))$ in $\N$ and $\q(1),\dots,\q(a(n))$ in $\N$ have been defined, with $Y_i = \Omega \setminus ( W_1 \cup \cdots \cup W_i )$ for all $1 < i \le a(n)$ and $\P(i) = \p(a(i))$ and $\Q(i) = \q(a(i))$ for all $1 \le i \le n$.
Define
\[
W_k
=
\{ x \in \Omega : x(k) = 7 \textup{ for all } 1 \le i < k \textup{ and } x(k) = 8 \}
=
[7^{k-1}8]
\]
for each $a(n) < k < a(n+1)$.
Note that $W_k$ was not removed at any earlier stage.
We have
\begin{equation}
\label{eqn:meas_bounded_removal}
\nu(W_k) = \frac{1}{t(k)}
\end{equation}
for all $a(n) < k < a(n+1)$.
Let
\[
Y_k
=
\Omega \setminus ( W_1 \cup \cdots \cup W_k )
=
Y_{a(n)} \setminus ( W_{a(n)+1} \cup \cdots \cup W_k )
\]
for all $a(n) < k < a(n+1)$.
Define $\p(k)$ to be the minimal $m \in \N$ with
\[
(S|Y_k)^m [0^{k-1}0] = [0^{k-1}1]
\]
and define $\q(k) = 0$ for all $a(n) < k < a(n+1)$. 

Next, let $\Q(n+1) = \q(a(n+1)) = 2 \P(n) - \Q(n)$ and put
\[
W_{a(n+1)}
=
\bigcup_{j=1}^{\Q(n+1)}
\left\{
x \in Y_{a(n+1)-1}
:
x(a(n+1)) \ge d(n+1) + 1
\textup{ and } (S|Y_{a(n+1)-1})^j(x) \in [0^{a(n+1)-1}]
\right\}
\]
which is shown in Figure~\ref{fig:unbounded_removal}.
We have
\begin{equation}
\label{eqn:meas_unbounded_removal}
\nu(W_{a(n+1)}) = \frac{(d(n+1)+1) \Q(n+1)}{t(a(n+1))}
\end{equation}

Let
\[
Y_{a(n+1)} = \Omega \setminus (W_1 \cup \cdots \cup W_{a(n+1)} )
\]
and define $\p(a(n)+1) \in \N$ to be minimal satisfying
\[
(S|Y_{a(n+1)})^{\p(a(n+1))}[0^{a(n+1)-1} 0] = [0^{a(n+1) - 1}1]
\]
and put $\P(n+1) = \p(a(n+1))$ to complete the inductive construction.
It is immediate from Figures~\ref{fig:boundedRemoval} and \ref{fig:unbounded_removal} that
\begin{equation}
\label{eqn:pn_bound}
(c(n) + 1) \p(n) \le t(n)
\end{equation}
for all $n \in \N$  and in particular, $t(n)$ is the number of levels of $Y$ at stage $n$ and $p(n)$ is morally the number of levels of $X$ at $p(n-1)$.

\begin{lemma}
\label{lem:removed_mass}
$\displaystyle\sum\limits_{k \in \N} \nu(W_k) < \frac{1}{8}$
\end{lemma}
\begin{proof}
First estimate
\[
\dfrac{(d(n)+1) \Q(n)}{t(a(n))}
=
\dfrac{d(n)+1}{c(a(n))+1} \cdot \dfrac{\Q(n)}{t(a(n)-1)}
\le
\dfrac{2\P(n-1)}{t(a(n)-1)}
\le
\dfrac{t(a(n-1))}{t(a(n)-1)}
=
\dfrac{1}{10^{a(n) - a(n-1)-1}}
\]
using \eqref{eqn:pn_bound}.
Then use \eqref{eqn:meas_bounded_removal} and \eqref{eqn:meas_unbounded_removal} to get
\[
\sum_{k \in \N} \nu(W_k)
\le
\sum_{k \in \N} \frac{1}{t(k)} + \sum_{n \in \N} \dfrac{(d(n)+1) \Q(n)}{t(a(n))}
\le \sum_{n \in \N} \dfrac{1}{10^n} + \dfrac{1}{10^{a(n+1) - a(n)-1}}
\]
which is less than $1/8$.
\end{proof}

By \cref{lem:removed_mass} the set 
\[
X = \Omega \Bigg\backslash \bigcup_{k \in \N} W_k
\]
has positive measure.
It is a closed subset of $\Omega$ as each $W_k$ is open.
The set
\[
G
=
\{ x \in X : S^n(x) \in X_0 \textup{ for infinitely many } n \in \N \textup{ and } S^{-n}(x) \in X_0 \textup{ for infinitely many } n \in \N \}
\]
has full measure in $X$ by Poincar\'e{} recurrence.
We define $T : X \to X$ by taking $T$ to be the first return map on $G$ and the identity map on $X \setminus G$.
Define $\mu(B) = \nu(B \cap X) / \nu(X)$ for every set $B \subset X$ of the form $C \cap X$ with $C$ a Borel subset of $\Omega$.
Our system is $(X,\mu,T)$. As $(\Omega,\nu,S)$ is ergodic $(X,\mu,T)$ is as well.

\begin{figure}[t]

\centering

\begin{tikzpicture}[yscale=0.5, line width=0.5mm]

\begin{scope}[shift={(-0.5,0)}]
\draw [decorate, decoration = {brace}] (0,-1) --  (0,8);
\node[align=right] at (-1.2,3.5) {$\P(n)$\\cylinders};
\end{scope}

\begin{scope}
\draw (0,8) -- (1,8);
\draw (0,7) -- (1,7);
\node at (0.5,6.2) {$\vdots$};
\draw (0,5) -- (1,5);
\draw (0,4) -- (1,4);
\draw (0,3) -- (1,3);
\draw (0,2) -- (1,2);
\node at (0.5,1.2) {$\vdots$};
\draw (0,0) -- (1,0);
\draw (0,-1) -- (1,-1);
\node at (0.5,-2) {$0$};
\end{scope}

\begin{scope}[shift={(1.5,0)}]
\node at (0.5,6) {$\cdots$};
\node at (0.5,1) {$\cdots$};
\end{scope}

\begin{scope}[shift={(3,0)}]
\draw (0,8) -- (1,8);
\draw (0,7) -- (1,7);
\node at (0.5,6.2) {$\vdots$};
\draw (0,5) -- (1,5);
\draw (0,4) -- (1,4);
\draw (0,3) -- (1,3);
\draw (0,2) -- (1,2);
\node at (0.5,1.2) {$\vdots$};
\draw (0,0) -- (1,0);
\draw (0,-1) -- (1,-1);
\node at (0.5,-2) {$d(n)$};
\end{scope}

\begin{scope}[shift={(4.5,0)}]
\draw[color=red] (0,8) -- (1,8);
\draw[color=red] (0.4,7.7) -- (0.6,8.3);
\draw[color=red] (0,7) -- (1,7);
\draw[color=red] (0.4,6.7) -- (0.6,7.3);
\node at (0.5,6.2) {$\vdots$};
\draw[color=red] (0,5) -- (1,5);
\draw[color=red] (0.4,4.7) -- (0.6,5.3);
\draw (0,4) -- (1,4);
\draw (0,3) -- (1,3);
\draw (0,2) -- (1,2);
\node at (0.5,1.2) {$\vdots$};
\draw (0,0) -- (1,0);
\draw (0,-1) -- (1,-1);
\node at (0.5,-2) {$d(n)+1$};
\end{scope}

\begin{scope}[shift={(6,0)}]
\node at (0.5,6) {$\cdots$};
\node at (0.5,1) {$\cdots$};
\end{scope}

\begin{scope}[shift={(7.5,0)}]
\draw[color=red] (0,8) -- (1,8);
\draw[color=red] (0.4,7.7) -- (0.6,8.3);
\draw[color=red] (0,7) -- (1,7);
\draw[color=red] (0.4,6.7) -- (0.6,7.3);
\node at (0.5,6.2) {$\vdots$};
\draw[color=red] (0,5) -- (1,5);
\draw[color=red] (0.4,4.7) -- (0.6,5.3);
\draw (0,4) -- (1,4);
\draw (0,3) -- (1,3);
\draw (0,2) -- (1,2);
\node at (0.5,1.2) {$\vdots$};
\draw (0,0) -- (1,0);
\draw (0,-1) -- (1,-1);
\node at (0.5,-2) {$2d(n)+1$};
\end{scope}

\begin{scope}[shift={(9,0)}]
\draw [decorate, decoration = {brace}] (0,8) --  (0,5);
\node[align=left] at (1.2,6.5) {$\Q(n)$\\cylinders};
\end{scope}

\end{tikzpicture}

\caption{The picture of $Y_{a(n)-1}$ at scale $a(n)$. The marked cylinders constitute $W_{a(n)}$.}
\label{fig:unbounded_removal}
\end{figure}
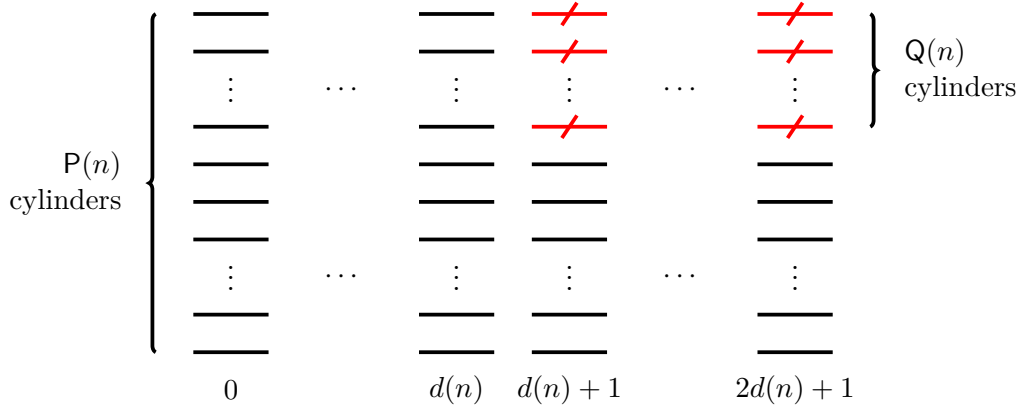

\subsection{Measures of cylinders}

We will note that our choices for $d(n)$ and $a(n)$ give the following properties.
\begin{enumerate}[label={\textbf{P\arabic*}.},ref={\textbf{P\arabic*}},leftmargin=5em]
\item
\label{p:tower_height}
$\P(n) \le t(a(n)-1)$
\item
\label{p:summing_gaps}
$\displaystyle \sum_{n=1}^\infty \dfrac{\P(n)}{\P(n+1)} < \dfrac{1}{10^5}$
\end{enumerate}

We will want often to study partial $T$ orbits as orbits of some $S|Y_n$. For convenience we define  for any $r \in \Z$ and any $b \in \N$
\[
\bad(b,r) = \bigcup_{a(m) > b} \{x \in X : (S|Y_b)^i(x) \in W_{a(m)} \textup{ for some } -|r| \le i \le |r| \}
\]
to consist of those points whose $S|Y_b$ orbits visit some finer $W_{a(m)}$ between times $-|r|$ and $|r|$.
When $x$ does not belong to $\bad(b,r)$ we have
\[
T^i(x) = (S|Y_b)^i(x) \textup{ for all } -|r| \le i \le |r|
\]
unless some $(S|Y_b)^i(x)$ belongs to some $[7^i 8]$ with $i \ge b$.

If a cylinder set $C$ of length $k$ does not belong to $W_1 \cup \cdots \cup W_k$ then $\mu$ assigns it positive measure.
In other words, any cylinder set $C$ not removed by scale $|C|$ has positive measure in $X$.
Slightly more is true: the following lemma will be used to prove $\mu(C) t(|C|)$ is bounded away from zero over such cylinder sets.

\begin{lemma}
\label{lem:cylinder_decay}
If $C \ne [7^{|C|}]$ is a cylinder not contained in $W_1 \cup \cdots \cup W_{|C|}$ then
\[
\nu(C \setminus X) \le \dfrac{4}{t(|C|)} \sum_{a(n) > |C|} \dfrac{\P(n-1)}{\P(n)}
\]
holds.
\end{lemma}
\begin{proof}
Fix a cylinder set $C \ne [7^{|C|}]$ disjoint from $W_1 \cup \cdots \cup W_{|C|}$.
We can only possibly have $C \cap W_j \ne \emptyset$ when $j = a(n) > |C|$ for some $n \in \N$.
For any such $n$, note that $C$ consists of $t(a(n)) / t(|C|)$ cylinders of length $a(n)$ and that at most
\[
\dfrac{2 (d(n) + 1) \Q(n)}{(2d(n)+1) \P(n)}
\le
\dfrac{2 \Q(n)}{\P(n)}
\le
\dfrac{4 \P(n-1)}{\P(n)}
\]
of them belong to $W_{a(n)}$.
Therefore
\[
\nu(C \cap W_{a(n)}) \le \dfrac{4}{t(|C|)} \dfrac{\P(n-1)}{\P(n)}
\]
from which the conclusion follows.
\end{proof}

\begin{lemma}
\label{lem:cylinder_size}
There is $\xi > 0$ such that
\[
\mu(C) \ge \dfrac{\xi}{t(|C|)}
\]
whenever $C$ is a cylinder set not contained in $W_1 \cup \cdots \cup W_{|C|}$.
\end{lemma}
\begin{proof}
If $C = [7^{|C|}]$ then
\[
\nu(C \cap W_n) \le \dfrac{1}{t(n)}
\]
whenever $n > |C|$ is a bounded scale.
Thus for every $C$ not contained in $W_1 \cup \cdots \cup W_{|C|}$ we have
\[
\nu(C \setminus X)
\le
\sum_{n > |C|} \dfrac{1}{t(n)} + \dfrac{4}{t(|C|)} \sum_{a(n) > |C|} \dfrac{\P(n-1)}{\P(n)}
\le
\dfrac{1}{t(|C|)} \left( \dfrac{1}{9} + \dfrac{4}{10^5} \right)
\]
after \cref{lem:cylinder_decay} and \ref{p:summing_gaps}.
Thus
\[
\xi = \dfrac{55}{63} \dfrac{1}{\nu(X)}
\]
suffices.
\end{proof}

\begin{lemma}
\label{lem:indep}
Let $D \subset \Omega$ be a cylinder whose largest defining index is $r$ and let $C \subset \Omega$ be a cylinder whose smallest defining index is $r+k$ for some $k \ge 1$. 
Further assume $D \cap X \neq \emptyset$ and $D \cap C \cap X \neq \emptyset$.
Then
\[
\frac{\mu(D\cap C \cap X)}{\mu(D)\mu(C)}\geq 1-\frac{1}{2^k}
\]
holds.
\end{lemma}

\begin{proof}

We first consider the case when $D$ has every index defined up to $r$ and the indices defining $C$ are from $r+1$ to $b$.
(This is a special version of the case $k=1$.)
Thus $C \cap D$ is a single cylinder defined at entries $1$ to $b$ that is non-empty. Thus $ C \cap D \subset Y_b$ and so $\mu(C \cap D)=\nu(X)^{-1}(\frac 1 {t(b)} -\nu(C \cap D \cap X^c))$.
Thus we compute 
\[
\nu(Y_b \setminus X)\leq \sum_{j>b} \frac{1}{t(j)} + \sum_{a(j)>b} \frac{4\P(j-1)}{\P(j)}<\frac{1}{8} \nu(C \cap D) = \dfrac{1}{8} \dfrac {1}{t(b)}
\]
so
\[
\mu(C\cap D) \geq \left( \frac{7}{8} \right)^2 \frac{1}{t(r)} \frac{1}{\nu(X)} \frac{t(r)}{t(b)} \frac{1}{\nu(X)} \geq \left(\dfrac{7}{8}\right)^2\mu(C) \mu(D)
\]
establishing the first case.

We now treat $D$ defined at all indices at most $r$ and $C$ is defined at all indices from $r+k$ to $b$. We consider $D \cap Y_{r+k-1},$ a union of $N\geq 8^k $ cylinders. Now 
\begin{equation}
\label{eq:D bound}
\left| \mu(D)-\frac{N}{\p(r+k+1)} \right|
\leq
\max \left\{
\nu(Y_{r+k-1}\setminus X)\nu(X)^{-1},
\frac{\nu(X)^{-1}}{\nu(Y_{r+k-1})^{-1}}
\right\}.
\end{equation}
Indeed we estimate the maximum amount that could be removed by going from $Y_{r+k-1}$ to $X$ and assume it is all removed from $D$ and on the other hand, if nothing is removed from $D$, how much we have to scale it to obtain its proportion of $X$ from its proportion of $Y_{r+k-1}$. Similarly 
\begin{equation}
\label{eq:C bound}
\left| \mu(C)-\frac{\p(r+k+1)}{\p(b+1)} \right|
\leq
\max \left\{ \nu(Y_{b}\setminus X)\nu(X)^{-1}, \frac{\nu(X)^{-1}}{\nu(Y_{b})^{-1}} \right\}.
\end{equation}
At most one of the $N$ cylinders defined at entries 1 to $r+k-1$ intersects $\bigcup_{j \notin \{a(k)\}}W(j)$.

For each of $N-1$ other cylinders,
we intersect them with $C$ and obtain cylinders 
defined at entries from $1$ to $b$ which do not intersect $W_j$ for $j \notin \{a(i)\}_{i \in \mathbb{N}}.$ Name these cylinders $E_1,...,E_{N-1}$ and note that
\[
\nu \left( \bigcup_{i=1}^{N-1}E_i \cap C \cap Y_b^c \right)
<
\sum_{a(j)>b}\frac 4 {t(b)}\dfrac{\P(j-1)}{\P(j)}< \frac 1 {2^{b+5}}\frac 1 {\p(b+1)}
\]
holds.
We obtain that
\[
\mu(C \cap D) \geq \frac{N-1}{\p(b+1)}-\frac{1}{2^{b+5}}\frac{1}{\p(b+1)}
\]
Since $N\geq 8^k$, combining this with \eqref{eq:D bound} and \eqref{eq:C bound} we have the claim. 

The general case is similar: treat $D$ as a union of $A$ cylinders defined on every index from 1 to $r+k-1$ and $C$ as a union of $B$ cylinders defined on every index from $r+k$ to $b$.

\end{proof}

\subsection{Expansion and reduction}
\label{subsec:expansion_reduction}
Given $r \in \Z$ we will analyze $T^r$ by expanding $r$ using a numeration system closely tied to the dynamics of $T$.
The numeration system is based on the heights $\p(j)$ of the left-most towers in the scale $j$ picture.
At unbounded scales we might instead have chosen to use $\p(j) - \q(j)$ in our numeration system as each height is witnessed by about half of the points in the system. Although we can only use one in our numeration system, we will in effect keep track of both in the technical parts of our arguments using the reductions in \cref{def:reduction}.
To define this expansion we need the following lemma.

\begin{lemma}
\label{lem:expansion_defined}
We have
\[
\left( \dfrac{c(n) + 1}{2} + 1 \right) \p(n) \le \p(n+1)
\]
for all $n \in \N$.
\end{lemma}
\begin{proof}
We have
\[
\p(n+1)
=
\begin{cases}
(c(n) + 1) \p(n) - 1 & n \textup{ a bounded scale} \\
(c(n) + 1) \p(n) - \q(n) (c(n)+1)/2 & n \textup{ an unbounded scale} \\
\end{cases}
\]
from which the statement follows.
\end{proof}
\begin{definition}
Fix $r$ in $\Z$.
By the \define{tower expansion} of $r$ we mean the expansion
\begin{equation}
\label{eqn:rigidity_expan}
r = \sum_{j=1}^{\scl(r)} e(j) \p(j)
\end{equation}
with $e(\scl(r)) \ne 0$ and $2|e(j)| \le c(j)+1$ chosen greedily.
Formally, this means the following: take
\[
\scl(r)
=
\max
\left\{
j \in \N : e \mapsto |r - e \p(j)| 
\textup{ is not minimized by } e = 0
\right\}
\]
with the convention that $\scl(0) = 0$.
Define $\lst(r) = e(\scl(r))$ to be the associated minimizing value, choosing $|\!\lst(r)|$ as large as possible whenever there is a choice. 
Since $\p(1) = 1$ both $1$ and $-1$ have tower expansions, and since $|r - \lst(r) \p(\scl(r))| < |r|$ one can define the tower expansion of any integer by induction on $|r|$.
Indeed letting $r'=r-e(r)J(r)$ we can inductively (using $|r'|<|r|$) assume $r'=\sum_{i=0}^{J(r')}e(i)p(i)$ and represent $r=e(r)p(J(r))+\sum_{i=0}^{J(r')}e(i)p(i)$. Since $J(r')<J(r)$ this gives an expansion for $r$.
We will refer to $\scl(r)$ as the \define{scale} of $r$. 
We will say that $r$ is \define{bounded} if $\scl(r)$ is a bounded scale, and that $r$ is \define{unbounded} if $\scl(r)$ is an unbounded scale.
\end{definition}

\begin{example}
We have $\p(1) = 1$, $\p(2) = 9$, $\p(3) = 89$, $p(4) = 889$ and $c(i) + 1 = 10$ for $i < a(1)$ so
\begin{gather*}
74 = 3  \p(1) -2 \p(2) +  \p(3) \\
150 = - \p(1) - 3 \p(2) + 2 \p(3)
\end{gather*}
are tower expansions.
\end{example}

We conclude this section with some basic results about our expansion.

\begin{lemma}
\label{lem:expansion_bound}
If $2|e(j)| \le c(j) + 1$ for all $1 \le j \le J$ then
\[
\left| \sum_{j=1}^J e(j) \p(j) \right| \le \left( \dfrac{c(J) + 1}{2} + 1 \right) \p(J)
\]
holds.
\end{lemma}
\begin{proof}
The proof is by induction.
One calculates
\[
\left| \sum_{j=1}^{K+1} e(j) \p(j) \right|
\le
\left( \dfrac{c(K) + 1}{2} + 1 \right) \p(K) + \dfrac{c(K+1) + 1}{2} \p(K+1)
\]
and applies \cref{lem:expansion_defined}.
\end{proof}

\begin{corollary}
\label{cor:obvioius_thing}
For all $r \in \Z$ we have $|r| \le \tfrac{3}{4} \p(\scl(r) + 1)$.
\end{corollary}
\begin{proof}
This follows from \cref{lem:expansion_defined}. If $J  = \scl(r)$ is a bounded scale then this follows because $\p(J+1)=10\p(J)-1$ and thus $(\frac{c(J)+1}2+1) \p(J)= 6 \p(J) < \frac{3}{4} \p(J+1)$.
Similarly if $J=a(k)$ is an unbounded scale, $\p(J+1) = 2(d(k)+1) \p(J) - d(k) \q(J)$ and so $(\frac{2d(k)+2}{2} + 1) \p(J) < \frac{3}{4} \p(J+1)$.
\end{proof}

\begin{lemma}
\label{lem:lot_red_bound}
For every $m \in \Z$ with $\scl(m)$ unbounded we have $2 |\lst(m) \q(\scl(m))| \le \p(\scl(m)-1)$.
\end{lemma}
\begin{proof}
Since $|\lst(m)| \le \dfrac{c(\scl(m))+1}{2}$ we can estimate
\[
|\lst(\scl(m)) \q(\scl(m))| \le \dfrac{c(\scl(m)) + 1}{2} \cdot 2 \p(\lambda(\scl(m))) \le 10^{-2^4} \p(\scl(m)-1)
\]
as required.
\end{proof}

When $r$ is at an unbounded scale a definite proportion of points $x$ with $x(\scl(r)) < c(\scl(r))/2$ satisfy $T^{\p(\scl(r))}(x) \approx x$ (see \ref{cor:pos red T left} and  \ref{cor:neg red T left}) and a definite proportion of points $x$ with $x(\scl(r)) > c(\scl(r))/2$ satisfy $T^{\p(\scl(r)) - \q(\scl(r))}(x) \approx x$ (see \ref{cor:pos red T right} and \ref{cor:neg red T right}).
Motivated by this, the next definition will be used to relate $T^r(x)$ for $r$ at unbounded scales to $T^{r'}(x)$ for $r'$ at smaller scales.

\begin{definition}
\label{def:reduction}
Define the \define{left} and \define{right reductions} by
\begin{align*}
\lft(r) &= r - \lst(r) \p(\scl(r)) \\
\rht(r) &= r - \lst(r) \p(\scl(r)) + \lst(r) \q(\scl(r))
\end{align*}
respectively whenever $\scl(r)$ is an unbounded scale. 
\end{definition}

A definite proportion of points $x$ satisfy $T^r(x) \approx T^{\lft(r)}(x)$ and a definite proportion of points $x$ satisfy $T^r(x) \approx T^{\rht(r)}(x)$ whenever $\scl(r)$ is unbounded.

\begin{lemma}
\label{lem:need it later}

Fix $r \in \Z$ with $\scl(r)$ unbounded.
\begin{itemize}

\item
$|\lft(r)| \le \p(\scl(r))$
\item
$|\rht(r)| \le \p(\scl(r)) - \q(\scl(r))$
\end{itemize}
\end{lemma}
\begin{proof}

For the first bullet, apply \cref{lem:expansion_bound} with $J = \scl(r) - 1$.
For the second bullet, apply \cref{lem:expansion_bound} with $J = \scl(r) - 1$ and \cref{lem:lot_red_bound}.
\end{proof}

The following lemmas describe how reductions and scales interact.

\begin{lemma}
\label{lem: reduction not both small}
Fix $m \in \Z$ with $\scl(m)$ an unbounded scale. 
We have
\[
\max\{\scl(\lft(m)), \scl(\rht(m))\} \ge \lambda(\scl(m))
\]
and if $\lft(m)$ and $\rht(m)$ are both at an unbounded scale then neither is at scale $\scl(m)$  and at least one of $\scl(\lft(m)) = \lambda(\scl(m))$ and $\scl(\rht(m)) = \lambda(\scl(m))$ holds. 
\end{lemma}
\begin{proof}
Since
\begin{equation}\label{eq:differnce in reduction}
\rht(m) - \lft(m) = \lst(m) \q(\scl(m)) = 2 \lst(m) \p(\lambda(\scl(m))) - \lst(m) \q(\lambda(\scl(m)))
\end{equation}
we have in general
\begin{equation}\label{eq:one correct}
\max\{|\rht(m)|,|\lft(m)|\}\geq \tfrac{3}{4} \p(\lambda(m))
\end{equation}
and so the larger of $\rht(m)$, $\lft(m)$ in absolute value
cannot be at a scale smaller than $\lambda(\scl(m))$, establishing the first claim.

We now assume that both $\lft(m),\rht(m)$ are at unbounded scales.
It is immediate that (in general) $\scl(\lft(m)) < \scl(m)$.
As $\lft(m)$ is assumed to be at an unbounded scale we must have $\scl(\lft(m)) \le \lambda(\scl(m))$.
That implies
\[
|\lft(m)| \le \dfrac{\p(\lambda(\scl(m))+1)}{2}
\]
and, as we have
\[
|\lst(m) \q(\scl(m))| \le \p(\scl(m)-1)
\]
from \cref{lem:lot_red_bound}, the triangle inequality gives
\[
2|\rht(m)| \le \p(\lambda(\scl(m))+1) + 2\p(\scl(m)-1) < \p(\scl(m))
\]
which means $\scl(\rht(m)) < \scl(m)$. 
As $\rht(m)$ is also assumed to be at an unbounded scale we have $\scl(\rht(m)) \le \lambda(\scl(m))$ as well.  Thus when they are both at an unbounded scale,  it must be that one of $\rht(m)$ and $\lft(m)$ is at scale $\lambda(\scl(m))$.
\end{proof}

\begin{lemma}
\label{lem:reduction LR}
Let $m \in \Z$ be at an unbounded scale $a(b)$.
At least one of $\scl(\rht(m)) < \scl(m)$ and $|\rht(m)|<\tfrac{3}{4}\p(\scl(m))$ is true. Also $\scl(\lft(m)) < \scl(m)$. 
\end{lemma}
\begin{proof}
This follows by the greedy algorithm to define $\p$ and the fact that $(2d(b)+2)\q(\scl(m))$ is much smaller than  $\tfrac{1}{8} \p(\scl(m))$.
\end{proof}

\subsection{Some dynamical properties}

We mention here briefly some dynamical properties of $(X,\mu,T)$ that follow quickly from the construction.

\subsubsection{Partial rigidity}
The task of proving $(X,\mu,T)$ is mild mixing is made difficult by partial rigidity in the system, which we remark upon here. 
Recall that $(X,\mu,T)$ is \define{partially rigid} if there is a constant $\delta > 0$ such that, for every measurable $A \subset X$ one has $\mu(A \cap T^n A) \ge \delta \mu(A)$ infinitely often.
Using Lusin's theorem, one can deduce partial rigidity from the following proposition.

\begin{proposition}
For all $\epsilon>0$
\[
\lim_{n \to \infty} \mu(\{x\in X: d(T^{\p(n)}x,x)<\epsilon\})\geq \frac{1}{9}
\]
holds.
\end{proposition}
Let $k_x:\mathbb{Z} \to \mathbb{Z}$ by $k_x(n)=m$ where $S^m(x) = T^{n}(x)$.
It is not hard to check that
\[
\nu(\{x \in X\subset \Omega:k_x(\p(n))=t(n-1)\})\geq \frac{1}{9}
\]
for all large enough $n$. 
Indeed, this follows from estimating $1-t(n-1)\nu(\bigcup_{k\geq n}W_k).$ Note that there are two cases, one where $n$ is at an unbounded scale and one where $n$ is at a bounded scale.

\subsubsection{Weak mixing}

\begin{proposition}
    $(X,\mu,T)$ is weakly mixing.
\end{proposition}
\begin{proof}
Observe that it suffices to show that for $k$ so that $k$ and $k-1$ are both bounded scales we have that there exists measurable $A(k)$, $B(k) \subset X$ with $\mu(A(k)),\mu(B(k))>\frac 1 {99}$, and 
\[
\underset{k \to \infty}{\lim}\sup\{d(T^{\p(k)}x,x):x\in A(k)\}=0=\underset{k \to \infty}{\lim}\sup\{d(T^{\p(k)}x,Tx):x\in B(k)\}
\]
holds.
(This is standard, but see for example \cite[Lemma~3.2]{MR4251942} for a closely related result with similar proof.) Let $A(k)=\{x:x(k)=1\} \setminus \bad(k,\p(k))$ and $B(k)=\{x:x(k)=8, x(k-1)<7\}\setminus \bad(k,\p(k))$. We leave it as an exercise to the reader to check that this meets the assumptions. 
\end{proof}

\section{The rank one property}
\label{sec:rank one}

We prove here that our system is rank one.

\begin{proof}[Proof of \cref{thm:rank one}]
Fix $A \subset X$ measurable and $\epsilon > 0$.
For each $n \in \N$ let $h(n) \in \N$ be minimal with the property that $(S|Y_{a(n)})^{h(n)} [8 7^{a(n)-1}]$ is the lexicographically largest cylinder of length $a(n)$ in $Y_{a(n)}$.
Certainly $h(n) < t(a(n))$.
All subsequent removals $W_k$ with $k \not\in \{a(m) : m > n \}$ are contained within $[7^{a(n)}]$.
Let
\[
F_n=[87^{a(n)-1}]\setminus \bad \bigg(a(n),h(n)\bigg) \subset X
\]
which is shown in Figure~\ref{fig:rank_one}.
We now bound the measure removed by removing $\bad(a(n),h(n))$.
For each $m > n$  we remove from $[8 7^{a(n)-1}]$ at most $d(m) + 1$ cylinders of length $a(m)$.
Thus
\begin{align*}
& \mu(X \setminus F_n \cup T^1 F_n \cup \cdots \cup T^{h(n)-1} F_n)
\\
\le
{}&
\dfrac{1}{\nu(X)} \left( 8 \P(n) \nu([0^{a(n)}]) + \sum_{m > n} h(n) \dfrac{d(m)+1}{t(a(m))} \right)
\\
\le
{}&
\dfrac{8}{\nu(X)} \dfrac{t(a(n)-1)}{t(a(n))} + \dfrac{2}{\nu(X)} \sum_{m > n} \dfrac{t(a(n)) d(m)}{t(a(m))}
\\
\le{}
&
\dfrac{8}{\nu(X)} \dfrac{1}{2d(n)+2} + \dfrac{2}{\nu(X)} \sum_{m > n} \dfrac{1}{10^{a(m) - a(n)-1}}
\end{align*}
where the term $8\P(n) \nu([0^{a(n)}])$ accounts for the cylinders of length $a(n)$ that precede $[87^{a(n)-1}]$ lexicographically and we have used $\P(n) \le t(a(n)-1)$.
We thus have
\[
\mu(F_n \cup T F_n \cup \cdots \cup T^{h(n)} F_n) \to 1
\]
as $n \to \infty$.

Choose $n$ so large that there exists cylinders $C_1,...,C_t\subset Y$ with defining coordinates $1,....,a(n)$ and
\begin{itemize}
\item
$\mu \Big( A \triangle (X \cap C_1 \cup \cdots \cup X \cap C_t) \Big) < \epsilon/100 \nu(X)$
\item
$8 \P(n) \nu([0^{a(n)}]) < \epsilon/10$
\item
$\mu(X \setminus F_n \cup T^1 F_n \cup \cdots \cup T^{h(n)-1} F_n) < \epsilon/2$
\end{itemize}
all hold.
Then the $\mu$ measure of the intersection of $A$ with the union of all cylinders that precede $[87^{a(n) - 1}]$ lexicographically is at most $\epsilon/10 \nu(X)$.
For $0 \le i < h(n)$ put $A_i = T^i F_n$ if
\[
(C_1 \cup \cdots \cup C_t) \cap T^i F_n \ne \emptyset
\]
and $A_i = \emptyset$ otherwise.
The set $A' = A_0 \cup \cdots \cup A_{h(n)-1}$ is measurable with respect to the $\sigma$-algebra generated by $\{ F_n, T F_n, \dots, T^{h(n)-1} F_n \}$ and $\mu(A' \triangle A) < \epsilon$.
\end{proof}

\begin{figure}[t]

\centering

\begin{tikzpicture}[yscale=0.5, line width=0.5mm]

\begin{scope}[shift={(-4.5,0)}]
\node at (0,4) {$87 \cdots 77$};
\node at (0,3) {$77 \cdots 77$};
\node at (0,2) {$67 \cdots 77$};
\node at (0,0) {$10 \cdots 00$};
\node at (0,-1) {$00 \cdots 00$};
\end{scope}

\begin{scope}[shift={(-3,0)}]
\draw (0,9) -- (1,9);
\node at (0.5,8.2) {$\vdots$};
\draw (0,7) -- (1,7);
\draw (0,6) -- (1,6);
\node at (0.5,5.2) {$\vdots$};
\draw (0,4) -- (1,4);
\draw (0,3) -- (1,3);
\draw (0,2) -- (1,2);
\node at (0.5,1.2) {$\vdots$};
\draw (0,0) -- (1,0);
\draw (0,-1) -- (1,-1);
\node at (0.5,-2) {$0$};
\end{scope}

\begin{scope}[shift={(-1.5,0)}]
\node at (0.5,8) {$\cdots$};
\node at (0.5,5) {$\cdots$};
\node at (0.5,1) {$\cdots$};
\end{scope}

\begin{scope}[shift={(0,0)}]
\draw (0,9) -- (1,9);
\node at (0.5,8.2) {$\vdots$};
\draw (0,7) -- (1,7);
\draw (0,6) -- (1,6);
\node at (0.5,5.2) {$\vdots$};
\draw (0,4) -- (1,4);
\draw[dashed] (0,3) -- (1,3);
\draw (0,2) -- (1,2);
\node at (0.5,1.2) {$\vdots$};
\draw (0,0) -- (1,0);
\draw (0,-1) -- (1,-1);
\node at (0.5,-2) {$7$};
\end{scope}

\begin{scope}[shift={(1.5,0)}]
\node at (0.5,8) {$\cdots$};
\node at (0.5,5) {$\cdots$};
\node at (0.5,1) {$\cdots$};
\end{scope}

\begin{scope}[shift={(3,0)}]
\draw (0,9) -- (1,9);
\node at (0.5,8.2) {$\vdots$};
\draw (0,7) -- (1,7);
\draw (0,6) -- (1,6);
\node at (0.5,5.2) {$\vdots$};
\draw (0,4) -- (1,4);
\draw (0,3) -- (1,3);
\draw (0,2) -- (1,2);
\node at (0.5,1.2) {$\vdots$};
\draw (0,0) -- (1,0);
\draw (0,-1) -- (1,-1);
\node at (0.5,-2) {$d(n)$};
\end{scope}

\begin{scope}[shift={(4.5,0)}]
\draw[color=red] (0,9) -- (1,9);
\draw[color=red] (0.4,8.7) -- (0.6,9.3);
\node at (0.5,8.2) {$\vdots$};
\draw[color=red] (0,7) -- (1,7);
\draw[color=red] (0.4,6.7) -- (0.6,7.3);
\draw (0,6) -- (1,6);
\node at (0.5,5.2) {$\vdots$};
\draw (0,4) -- (1,4);
\draw (0,3) -- (1,3);
\draw (0,2) -- (1,2);
\node at (0.5,1.2) {$\vdots$};
\draw (0,0) -- (1,0);
\draw (0,-1) -- (1,-1);
\node at (0.5,-2) {$d(n)+1$};
\end{scope}

\begin{scope}[shift={(6,0)}]
\node at (0.5,8) {$\cdots$};
\node at (0.5,5) {$\cdots$};
\node at (0.5,1) {$\cdots$};
\end{scope}

\begin{scope}[shift={(7.5,0)}]
\draw[color=red] (0,9) -- (1,9);
\draw[color=red] (0.4,8.7) -- (0.6,9.3);
\node at (0.5,8.2) {$\vdots$};
\draw[color=red] (0,7) -- (1,7);
\draw[color=red] (0.4,6.7) -- (0.6,7.3);
\draw (0,6) -- (1,6);
\node at (0.5,5.2) {$\vdots$};
\draw (0,4) -- (1,4);
\draw (0,3) -- (1,3);
\draw (0,2) -- (1,2);
\node at (0.5,1.2) {$\vdots$};
\draw (0,0) -- (1,0);
\draw (0,-1) -- (1,-1);
\node at (0.5,-2) {$2d(n)+1$};
\end{scope}

\end{tikzpicture}

\caption{The picture of $Y_{a(n)-1}$ at scale $a(n)$. The marked cylinders make up $W_{a(n)}$. The dashed cylinder contains all $W_j$ with $j > a(n)$ and $j$ a bounded scale. A subset of $[87^{a(n)-1}]$ will be used as the set $F$ exhibiting $T$ as rank one.}
\label{fig:rank_one}
\end{figure}
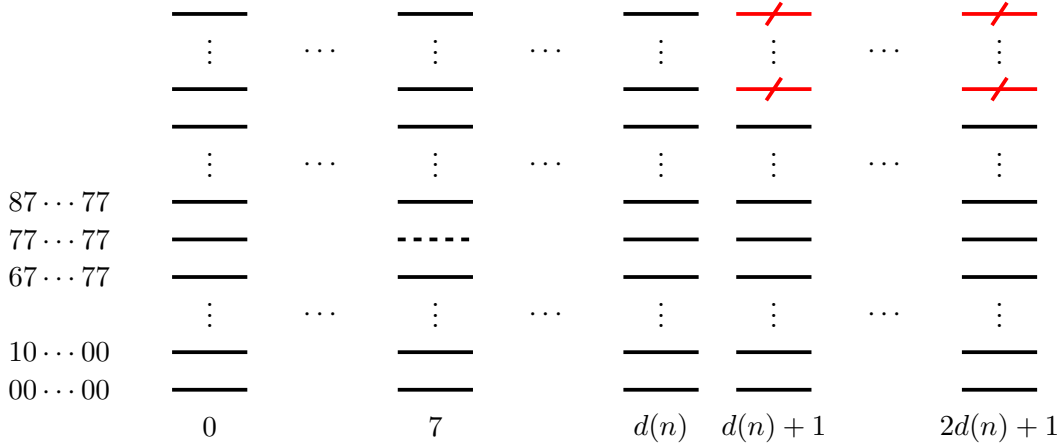

\begin{remark}
\label{rem:rank two}
It is natural to try to determine whether our procedure can be modified to construct a mildly mixing rank 1 transformation $T$ with the property that $T^n$ has rank two for all $n\geq 2$.
(Note that $T$ being mild mixing precludes $T^n$ from being rank one: indeed $T$ is in its centralizer and were $T^n$ rank one it could be concluded that $T$ is rigid by the weak closure theorem. We thank Mariusz Lemanczyk for pointing this out.)
Indeed, one could try to modify our construction (perhaps by removing some extra levels at certain bounded scales and changing the definition of $\Q(n)$) so that for every $n$ there are infinitely many $k$ so that at the unbounded scales $a(k)$ we have $\P(k)$ and $\P(k)-\Q(k)$ both relatively prime to $n$. One could use the following definitions.
$$F_k=\bigcup_{i=0}^{n-1}(S|Y_{a(k)})^i[87^{a(k)-1}]\setminus \bad(a(k),\p(k)+2)$$
$$F_k=\bigcup_{i=0}^{n-1}(S|Y_{a(k)})^i[87^{a(k)-2}d(k)+2]\setminus \bad(a(k),\p(k)+2).$$
Letting $$h=\min\{i>0:(S|Y_{a(k)})^{in}F \cap [0^{a(k)-1}d(k)] \neq \emptyset\}-1$$ and 
$$h'=\min\{i>0:(S|Y_{a(k)})^{in}F' \cap [0^{a(k)}] \neq \emptyset\}-1$$ we see our transformation is rank 2. 

Having done the modification to the construction, one still needs to check that the proof of mild mixing goes through.
We do not pursue this here.
\end{remark}

\section{An ergodic joining}
\label{sec:ergodic joining}

The purpose of this section is to construct an ergodic self-joining of $(X,\mu,T)$ that is neither the product measure nor an off-diagonal joining. The construction is an application of \cite[Proposition~3.1]{MR4269425}; to prepare for its application in \cref{subsec:apply_CE} we give in \cref{subsec:two_toweres} a technical description of towers for $(X,\mu,T)$.

\subsection{A tower}
\label{subsec:two_toweres}

In the picture at scale $a(k)$ we have a tower
\[
\mathcal{T} = \bigcup_{n=0}^{(d(k) + 1)\P(k)-1} (S|Y_{a(k)})^n [0^{a(k)}] \subset Y_{a(k)}
\]
of height $(d(k)+1) \P(k)$ for the transformation $S|Y_{a(k)}$.
We will wish often to think of $\mathcal{T} \cap X$ as a tower for $T$ on which $T^n$ and $S|Y_{a(j)}^n$ agree for all $-2 \P(k) \le n \le 2 \P(k)$ for most points $x$ in $\mathcal{T} \cap X$.
In this subsection we produce subsets $J(k)$ of $[0^{a(k)}]$ that are the bases of towers for $T$ with such properties.
We will do so by removing ``columns'' from the tower.

\begin{definition}
By a \define{column} at scale $i$ we mean a subset of $\Omega$ defined by fixing values for all coordinates $i \le n \le I$. We call $I \ge i$ the \define{rank} of the column. The \define{base} of a column is its intersection with $[0^{i-1}]$.
\end{definition}

For example $\{ x \in \Omega : x(4) = 2, x(5) = 3 \}$ is a column at scale 4 and rank 5 while
\[
\{ x \in \Omega : 5 \le x(a(j)) \le 10 \}
\]
is a union of six columns at scale $a(j)$ and rank $a(j)$.

\begin{proposition}
\label{prop:joining_prep}
For every $k \in \N$ there are sets $J(k), U(k) \subset X$ and a constant $c>0$ with the following properties.
\begin{enumerate}[label={\textbf{\textup{PJ\arabic*}}.},ref={\textbf{PJ\arabic*}},leftmargin=5em]
\item
\label{jt:return}
$\min \{ n \in \N : T^n(x) \in J(k) \} \ge (d(k)+1)(2\P(k) - \Q(k))-3$ for every $x \in J(k)$
\item
\label{jt:shrinking}
$(d(k)+1) \P(k) \sum\limits_{j > k} \mu(J(j)) \to 0$
\end{enumerate}
Writing
\[
A(k) = \bigg(\bigcup_{n=0}^{(d(k)+1)\P(k)-1} T^n J(k)\bigg)
\bigg\backslash U(k)
\qquad \textup{and} \qquad
B(k) = \bigg(X \setminus A(k)\bigg) \setminus U(k)
\]
for every $k \in \N$ we have all the following properties.
\begin{enumerate}[label={\textbf{\textup{PJ\arabic*}}.},ref={\textbf{PJ\arabic*}},leftmargin=5em]
\setcounter{enumi}{2}
\item
\label{jt:left}
$T^n(x) = (S|Y_{a(k)})^n(x)$ for every $x \in A(k)$ and every $-2 \P(k) \le n \le 2 \P(k).$
\item \label{conc:same left} For every $x \in A(k)$ we have $(S|Y_{a(k)})^{\P(k)}(x) \stackrel{a(k)-1}{=}x$.
\item
\label{jt:right}
$T^n(x) = (S|Y_{a(k)})^n(x)$ for every $x \in B(k)$ and every $-2\P(k) + 2\Q(k) \le n \le 2\P(k) - 2\Q(k).$
\item \label{conc:same right} For every $x \in B(k)$ we have $(S|Y_{a(k)})^{\P(k)-\Q(k)}(x)  \stackrel{a(k)-1}{=}x$.
\item
\label{jt:discard}
$\mu(U(k)) \le \dfrac {15} {2^k}$
\item
\label{jt:proportion}
$\mu(A(k)) \ge c$ and $\mu(B(k)) \ge c$
\end{enumerate}
\end{proposition}
\begin{proof}
For each $k \in \N$ define
\[
V(k) = \bad \bigg(a(k),(d(k)+1)(2\P(k)-\Q(k))-1\bigg)
\]
and put $J(k) = [0^{a(k)}] \setminus V(k)$.

Note that $J(k) \subset X$ for all $k$.
For every $x \in J(k)$ and every $0 \le n \le (d(k)+1)(2\P(k)-\Q(k))-3$ we have
\begin{equation}
\label{eqn:tower_discrep}
T^n(x) = (S|Y_{a(k)})^n(x) \textup{ or } T^n(x) = (S|Y_{a(k)})^{n+1}(x) 
\end{equation} and
\begin{equation}
\label{eq:equal}
(S|Y_{a(k)})^j(x) \notin \bigcup_{\substack{j>a(k)\\j \notin\{a(i)\}}}W_j
\Rightarrow
T^n(x)=(S|Y_{a(k)})^n(x).
\end{equation}
Indeed, if $T^n(x) \ne (S|Y_{a(k)})^n(x)$ for some $n$ in the above range then
\[
(S|Y_{a(k)})^j(x) \in \bigcup_{i>a(k)}W_i
\]
for some $0 < j \le n$.
But as $x \notin V(k)$ this restriction becomes
\[
(S|Y_{a(k)})^j(x) \in \bigcup_{\substack{i>a(k)\\i \textup{ bounded}}} W_i
\]
and
\[
\min
\left\{
\ell>0
:
\begin{aligned}
(S|Y_{a(k)})^\ell x \in W_j \text{ for some } x \in W_i \setminus V(k)\\ \text{ with }i\neq j, \, i,j\in \{a(k)+1,a(k)+2,...\}\setminus \{a(r)\}
\end{aligned}
\right\}
\ge \p(a(k)+1)-1
\]
so for each $x \in J(k)$ 
\[
\left|
\left\{
0\leq j\leq n:(S|Y_{(a(k)})x \in \bigcup_{\ell>a(k)}W_\ell \right\} \right|
\leq 1
\text{ and }
\left| \left\{ 0\geq j\geq -n:(S|Y_{(a(k)})x \in \bigcup_{\ell>a(k)}W_\ell \right\} \right| \leq 1
\]
which establishes \ref{jt:return} because $(d(k)+1)(2\P(k)-\Q(k))$ is the first return time to $[0^{a(k)}]$ for $S|Y_{a(k)}$.

For \ref{conc:same left} and \ref{conc:same right} note that the definition of $A(k)$ gives that $x \in A(k)$ implies $0\leq x(\ell)\leq d(k)+1$ and   $y \in A(k)^c\supset B(k)$ implies $0\leq y(\ell)\leq 2d(k)+1$, consulting Figure \ref{fig:unbounded_removal} gives these claims.

For \ref{jt:shrinking} note that $\mu(J(k)) \le 1/\nu(X) t(a(k)+1)$ giving
\begin{equation}\label{eq:est JT}
(d(k)+1)\P(k) \sum_{j > k} \mu(J(j))
\le
\dfrac{1}{\nu(X)} \sum_{j > k} \dfrac{(d(k)+1)\P(k)}{t(a(j)+1)}
\le
\dfrac{1}{\nu(X)} \sum_{j > k} \dfrac{1}{10^{a(j) - a(k)}}
\end{equation}
because $(d(k)+1)\P(k) \le t(a(k))$ for all $j > k$.

To get \ref{jt:left} and \ref{jt:right} we define the sets $U(k)$ in terms of columns.
First define
\begin{equation}\label{eq:U}
U_0(k)=\bigg\{x\in \Omega:x(a(k)\notin \{0,1,5,6,7,8,9,d(k)-1,...,d(k)+2,2d(k),2d(k)+1\}\bigg\}
\end{equation}
which is a union of columns at scale $a(k)$. We have
\[
\nu(U_0(k)) = \dfrac{13 \P(k)}{t(a(k))} \le \dfrac{13}{d(k)}
\]
and therefore
\begin{equation}
\label{eq:JT est 2}
\mu(U_0(k) \cap X) \le \dfrac{13}{d(k) \nu(X)} \le \dfrac{15}{2^k}
\end{equation}
which goes to zero as $k \to \infty$.

Now, if $x \in X \setminus \bigg(U_0(k) \cup V(k)\bigg)$ 
then for all $0\leq |n|\leq \P(k)$ because $x \notin V(k)$ and $(T^nx)(a(k))\neq 7$, as in 
\eqref{eq:equal}, $(S|Y_{a(k)})^nx=T^nx$.

Setting $U(k)=U_0(k) \cup V(k)$
we have \ref{jt:left} and \ref{jt:right}.
Using \eqref{eq:JT est 2} we have 
\[
\mu(U(k))\leq \frac{14}{2^k}+\frac 1 {\nu(X)} 2(d(k)+1)(2\P(k)-\Q(k))-1)+1)\nu(\bigcup_{j>k}W_{a(j)})\leq \dfrac {15} {2^k}
\]
which is \ref{jt:discard}.

Lastly, \ref{jt:proportion} is a consequence of $\{ x \in Y_{a(k)} : x(k) \le d(k) \}$ and $\{x \in Y_{a(k)} : x(k) > d(k) \}$ both having measure at least $\tfrac{1}{4}$ and \ref{jt:discard}.
\end{proof}

\subsection{Building the joining}
\label{subsec:apply_CE}

In this section we prove \cref{thm:ergodic joining}.
We begin by recalling that the Kantorovich-Rubenstein metric on the space $\meas_1(V)$ of Borel probability measures on a compact metric space $V$ is defined by
\[
\metric_\KR(\eta,\xi) = \sup \left\{ | \eta(f) - \xi(f) | : f \textup{ is 1-Lipschitz on } X \textup{ with respect to } \metric|_{X\times X} \right\}
\]
for all $\eta, \xi$ in $\meas_1(V)$.

The joining $\lambda$ will be built using a minor modification of \cite[Proposition~3.1]{MR4269425} (in the case $d=2$). We highlight the differences in bold and in the appendix we prove the version of \cite[Corollary 3.3]{MR4269425}, which is the only part of the proof that requires significant modification. Note that \cref{ce:kr}  is a strengthening of \cite[Proposition 3.1 (B)]{MR4269425}, the corresponding condition in \cite{MR4269425}, and this is used for \eqref{eq:CEchanged} in the appendix.

\begin{proposition}
\label{prop:ce}
Fix sequences
\begin{itemize}
\item
$k \mapsto J(k)$ of \textbf{subsets of} $X$
\item
$k \mapsto U(k)$ of measurable subsets of $X$
\item
$k \mapsto r(k)$ of natural numbers
\item
$k \mapsto \alpha(k)$ and $k \mapsto \gamma(k)$ of \textbf{integers}
\item
$k \mapsto \epsilon(k)$ of positive real numbers
\end{itemize}
and set
\[
A(k) = \bigcup_{i=0}^{r(k)-1} T^i (J(k)) \bigg\backslash U(k)
\qquad
B(k) = (X \setminus A(k)) \setminus U(k)
\]
and let $\nu_i$ on $X \times X$ be the push-forward of $\mu$ under the map $x \mapsto (x, T^i(x))$. Assume all of the following.
\begin{enumerate}[label={\textbf{\textup{J\arabic*}}.},ref={\textbf{J\arabic*}},leftmargin=5em]
\item
\label{ce:proportion}
There is $c > 0$ such that for all $k \in \N$ one has $\mu(A(k)) \ge c$ and $\mu(B(k)) \ge c$.
\item
\label{ce:return}
For all $k \in \N $ the minimal return time of $T$ on $J(k)$ is at least $\tfrac{3}{2} r(k)$.
\item
\label{ce:discard}
For all $k \in \N$ we have $\mu(U(k)) < \epsilon(k)$.
\item
\label{ce:shrinking}
$\lim\limits_{k \to \infty} r(k) \sum\limits_{j > k} \mu(J(j)) = 0$
\item
\label{ce:error}
The sequence $k \mapsto \epsilon(k)$ is non-increasing and summable.
\item
\label{ce:nearness}
For every $k \in \N$ and every $x \in A(k)$ one has
\begin{equation}
\label{ce:nearness_1}
\metric(T^{\alpha(k)} x, T^{\gamma(k-1)} x) < \epsilon(k)
\qquad \text{and} \qquad 
\metric(T^{\gamma(k)} x, T^{\alpha(k-1)} x) < \epsilon(k)
\end{equation}
and for every $k \in \N$ and every $x \in B(k)$ one has
\begin{equation}
\label{ce:nearness_2}
\metric(T^{\alpha(k)} x, T^{\alpha(k-1)} x) < \epsilon(k)
\qquad\text{and}\qquad
\metric(T^{\gamma(k)} x, T^{\gamma(k-1)} x) < \epsilon(k)
\end{equation}
\item
\label{ce:kr}
For every $k \in \N$ the estimates
\begin{align*}
\metric_\KR \left( \nu_{\alpha(k)}, \dfrac{1}{L} \sum_{i=1}^L (T \times T)^i \delta(x,T^{\alpha(k)} x) \right) &{} < \epsilon(k) \\
\metric_\KR \left( \nu_{\gamma(k)}, \dfrac{1}{L} \sum_{i=1}^L (T \times T)^i \delta(x,T^{\gamma(k)} x) \right) &{} < \epsilon(k)
\end{align*}
both hold for all $x \in X$ and all $L \ge \tfrac{1}{9} \epsilon(k)r(k+1)$.
For every $k \in \N$ the estimates
\begin{align}
\label{eq:ce changed}\mu \left( \, \bigcup_{9L \ge r(k+1)} \left\{ x \in X : \metric_\KR \left( \nu_{\alpha(k)}, \dfrac{1}{L} \sum_{i=1}^L (T \times T)^i \delta(x,T^{\alpha(k)} x) \right) > \epsilon(k) \right\} \right) &{} < \epsilon(k) \\
\mu \left( \, \bigcup_{9L \ge r(k+1)} \left\{ x \in X : \metric_\KR \left( \nu_{\gamma(k)}, \dfrac{1}{L} \sum_{i=1}^L (T \times T)^i \delta(x,T^{\gamma(k)} x) \right) > \epsilon(k) \right\} \right) &{} < \epsilon(k)
\end{align}
both hold.
\end{enumerate} 
Then
\[
\lim_{j \to \infty} \nu_{\alpha(j)} = \lim_{j \to \infty} \dfrac{\nu_{\alpha(j)} + \nu_{\gamma(j)}}{2} = \lim_{j \to \infty} \nu_{\gamma(j)}
\]
and in particular the above limits exist. Moreover the common limiting value $\lambda$ is ergodic for $T \times T$ and there is $C > 0$ such that
\[
\metric_\KR \left( \lambda, \dfrac{\nu_{\alpha(k)} + \nu_{\gamma(k)}}{2} \right) \le C \sum_{i > k} \epsilon(i)
\]
for all $k \in \N$.
\end{proposition}

In order to apply \cref{prop:ce} to establish \cref{thm:ergodic joining} the central assumption to establish is \ref{ce:nearness}.
The times $\alpha(k)$ and $\gamma(k)$ will be chosen based on the picture at the unbounded scale $a(k)$.
The set $A(k)$ will consist of points in the left tower at scale $a(k)$ while the set $B(k)$ will consist of points in the right tower at scale $a(k)$.
Indeed, if $x \in A(x)$, then by \eqref{eq:U} we have $x(a(k))\in [2,d(k)-2]$ and so $T^n(x)$ is in the left tower at scale $a(k)$ for all $0\leq |n|\leq \P(k)$. Analogously for $x \in B(k)$ and the right tower.
Recall that $\P(k)$ is the height of the left tower and $\P(k) - \Q(k)$ is the height of the right tower at scale $a(k)$.
Taking
\[
\alpha(0) = 1
\qquad
\alpha(1) = \P(1)
\qquad
\alpha(2) = \P(2) - \P(1) + 1
\qquad
\alpha(3) = \P(3) - \P(2) + \P(1)
\]
and
\[
\gamma(0) = 0
\qquad
\gamma(1) = - \P(1) + 1
\qquad
\gamma(2) = - \P(2) + \P(1)
\qquad
\gamma(3) = -\P(3) + \P(2) - \P(1) + 1
\]
and using the inductive definition $\Q(n+1) = 2\P(n) - \Q(n)$ we have for example that $T^{b(3)} \approx T^{d(2)}$ and $T^{d(3)} \approx T^{b(2)}$ for many points in $A(3)$ and $T^{b(3)} \approx T^{b(2)}$ and $T^{d(3)} \approx T^{d(2)}$ for many points in $B(3)$.

\begin{proof}[Proof of \cref{thm:ergodic joining}]
We will apply \cref{prop:ce}.
Let the constant $c > 0$ and the sequences $k \mapsto J(k)$ and $k \mapsto U(k)$ be those from \cref{prop:joining_prep}.
Define $r(k) = (d(k)+1) \P(k) -1$  and put $\epsilon(k) = 15 \cdot 2^{-k}$ 
With these choices the sets $A(k)$ and $B(k)$ in the hypothesis of \cref{prop:ce} are the same as the sets $A(k)$ and $B(k)$ in the statement of \cref{prop:joining_prep}.
In particular we get \ref{ce:proportion} from \ref{jt:proportion} and \ref{ce:return} from
\[
3(d(k)+1) \P(k) \le 4(d(k)+1) (\P(k) - \Q(k))
\]
and \ref{jt:return}.
Our choice of $\epsilon(k)$ gives \ref{ce:discard} immediately and \ref{ce:shrinking} is just \ref{jt:shrinking}.
We get \ref{ce:error} from \ref{jt:discard}.

To establish \ref{ce:nearness} and \ref{ce:kr} we need to define the sequences $\alpha(k)$ and $\gamma(k)$.
Put $\alpha(0) = 1$ and $\gamma(0) = 0$ and inductively define
\[
\alpha(k+1) = \P(k+1) + \gamma(k)
\qquad\text{and}\qquad
\gamma(k+1) = - \P(k+1) + \alpha(k)
\]
for all $k \in \N$.
One can check by induction that
\begin{equation}
\label{eqn:joining_tower_relations}
\alpha(k) + \Q(k) = \alpha(k-1) + \P(k)
\textup{ and } 
\gamma(k) + \P(k) = \Q(k) + \gamma(k-1)
\end{equation}
for all $k \in \N$.

We verify that \ref{ce:nearness} holds.

We calculate
\[
T^{\alpha(k)}(x) = T^{\P(k) + \gamma(k-1)}(x) \stackrel{a(k)-1}{=} T^{\gamma(k-1)}(x)
\Rightarrow
\metric(T^{\alpha(k)}(x), T^{\gamma(k-1)}(x)) \le \epsilon(k)
\]
and
\[
T^{\gamma(k)}(x) = T^{-\P(k) + \alpha(k-1)}(x) \stackrel{a(k)-1}{=} T^{\alpha(k-1)}(x)
\Rightarrow
\metric(T^{\gamma(k)}(x), T^{\alpha(k-1)}(x)) \le \epsilon(k)
\]
using \ref{jt:left} and \ref{conc:same left}, which gives \eqref{ce:nearness_1}.
Similarly, for all $x \in B(k)$ we have
\[
T^{\alpha(k)}(x) \stackrel{(k)-1}{=} T^{\alpha(k) - \P(k) + \Q(k)}(x) = T^{\alpha(k-1)}(x)
\Rightarrow
\metric( T^{\alpha(k)}(x), T^{\alpha(k-1)}(x)) < \epsilon(k)
\]
and
\[
T^{\gamma(k)}(x) \stackrel{a(k)-1}{=} T^{\gamma(k) + \P(k) - \Q(k)}(x) = T^{\gamma(k-1)}(x)
\Rightarrow
\metric( T^{\gamma(k)}(x), T^{\gamma(k-1)}(x)) < \epsilon(k)
\]
using \ref{jt:right} and \ref{conc:same right}.
This establishes \eqref{ce:nearness_2} and thus \ref{ce:nearness}.

Lastly \ref{ce:kr} follows from ergodicity of $(X,\mu,T)$ after passing to a subsequence of $k \in \N$.
\end{proof}

\section{Mild mixing}
\label{sec:mild mixing}

In this section we will prove \cref{thm:mild mixing}.
Given a non-constant function $f$ in $\ltwo(X,\mu)$ we wish to prove that no sequence $\textbf{\textit{r}}$ in $\Z$ with $|r(n)| \to \infty$ satisfies $\| f - T^{r(n)} f \| \to 0$.

In \cref{subsec:buddies} we define the essential notion of ``buddies'' and give in \cref{thm:getting_constant weak} a criterion for a non-constant $f \in \ltwo(X,\mu)$ not to have $\textbf{\textit{r}}$ as a rigidity sequence.
The thrust of the criterion is that the existence of a uniformly positive proportion of $(r(\tau),\tau)$ buddies as $\tau \to \infty$ precludes $f$ from being rigid along $\textbf{\textit{r}}$.

Thus, given $\textbf{\textit{r}}$ we need to be able to produce buddies. There are certain situations where producing buddies is relatively straightforward. Those are the situations where we do not need to carry out the reduction procedure mentioned in the introduction. In those situations are able to instead produce ``friendly pairs'' which are defined in \cref{subsec:buddies from pairs}.

In fact, we would prefer to work with friendly pairs throughout and avoid altogether the notion of buddies.
However, when using the reduction procedure (see \cref{subsec:expansion_reduction}) to prove \cref{prop:reduction} we are unable to work directly with friendly pairs.

We give in \cref{subsec:reduction summary} a technical summary of the reduction procedure. In \cref{subsec:reduction once} we given some preparatory results describing what happens when we apply the reduction procedure once.
The main technical result is proved in \cref{subsec:reduction main}.
Finally, the proof of \cref{thm:mild mixing} is given in \cref{subsec:mild mixing}.

\subsection{Buddies}
\label{subsec:buddies}

Essential to our proof is the following definition.

\begin{definition}[Buddies]
\label{def:weak_friends}
Fix $r,j \in \N$.
We will say that $(y,z)$ in $X \times X$ are \define{$(r,j)$ buddies} if the following properties all hold.
\begin{enumerate}
\item
\label{cond:start same 2}
 $y \stackrel{j-1}{=} z$

\item
\label{cond:differ by 1}
There exists $0 < |i| <9$ so that
 $(S|Y_{j-1}^i \circ T^r)(y) \stackrel{j-1}{=} (T^r)(z)$ 

\end{enumerate}
\end{definition}

\begin{remark}
In our argument, we will have $|i|=1$. However, in related arguments, $|i|$ will often be in a larger range (see e.g. \cite{MR4251942}). As such we phrase things more generally to have Theorem \ref{thm:getting_constant weak} in a more general situation.
\end{remark}

The points $y$ and $z$ are $(r,j)$ buddies if they are initially close (in that they agree in the first $j-1$ coordinates) and through $r$ iterations of $T$  differ by at least one and at most eight iterations of $S|Y_{j-1}$.

\begin{example}
Fix $k \in \N$ a bounded scale.
When $y \in X \cap [6 7^{k-2} 7]$ with $S(y) = T(y)$ and $z \in X \cap [6 7^{k-2}8]$ with $S^2(y) = T(y)$ the points $y$ and $z$ are $(1,k)$ buddies.
\end{example}

We record here a lemma about buddies that will be  needed later.

\begin{lemma}
\label{lem:buddies through reduc}
Assume $r,u\in \mathbb{N}$, $j<J\in \mathbb{N}$ and $y,z\in X$ so that
\begin{enumerate}
\item $T^r(y) \stackrel{J-1}{=} T^u(y)$
\item $T^r(z) \stackrel{J-1}{=} T^u(z)$
\item $(y,z)$ are $(r,j)$-buddies,
\end{enumerate}
then $(y,z)$ are $(u,j)$-buddies.
\end{lemma}
\begin{proof}
The first condition of being $(a,j)$ buddies is independent of $a$ and so we just need to check the second condition. The second condition of being $(a,j)$-buddies depends only on the first $j-1$ coordinates of $T^a$. Indeed, $S|Y_{j-1}$ depends only on these coordinates. So by condition 1.\ and 2.\ above, $(y,z)$ being $(r,j)$-buddies and being $(u,j)$-buddies is equivalent.
\end{proof}

The following theorem is our main criterion for non-rigidity of $f \in \ltwo(X,\mu)$ along a sequence $\textbf{r}$ in $\Z$.

\begin{theorem}
\label{thm:getting_constant weak}
Let $f:X \to \mathbb{C}$ be a non-constant $\ltwo(X,\mu)$ function and let $\textbf{r}$ be a sequence in $\Z$ with $|r(n)| \to \infty$.
Assume there exists $\rho>0$ so that for all $\tau \in \mathbb{N}$ there exists 
$ i_0 \in \mathbb{N}$ so that for all $i\geq i_0$ there exists 
a Borel set $\Gamma_i \subset X$ and a measure preserving function $\phi_i : X \to X$ with
\begin{itemize}
\item
$\mu(\Gamma_i) > \rho$
\item
for all $z \in \Gamma_i$ we have that $(\phi_i(z),z)$  are $(r(i),\tau)$-buddies
\end{itemize}
all hold.
Then $\textbf{r}$ is not a rigidity sequence for $f$.
\end{theorem}
\begin{lemma}\label{lem:same power} 
For all $\epsilon>0$ and $B\in \mathbb{N}$ there exists $\tau$ so that if $r \in \mathbb{Z}$, $\Gamma$ is a Borel set and $\phi:\Gamma \to X$ is a measure preserving function satisfying $(\phi(z),z)$ are $(r,\tau)$-buddies for then 
\[
\mu
\left(
\left\{
y\in \Gamma
:
\begin{aligned}
\exists \ell \text{ with }0<|\ell|<10 \text{ so that for all } \\ -B\leq a\leq B \text{ we have } \metric \Big( (T^{r+a} \circ \phi_i)(y), T^{r+a+\ell}(y) \Big) > \epsilon
\end{aligned}
\right\}
\right)
<\epsilon.
\]
\end{lemma}
\begin{proof}
   
The definition of $(y,z)$-buddies gives that for any $\delta>0$ there exists $j_0$ so that for all $j'\geq j_0$ we have that for every $(y,z)$ which are $(r,j')$-buddies we have that there exists $0<|k|<10$ so that 
\begin{equation}\label{eq:buddies close}
\min_{0<|k|<10}\metric\left( \Big(T^{r} \circ \phi_i\Big)(y) ,\Big((S|Y_{j'-1} )^k \circ T^{r}\Big)(y)\right)<\delta.
\end{equation}

Because $T^k:X \to X$ is measurable for all $k \in \mathbb{Z}$, by Lusin's theorem and the fact that continuous functions on compact sets are uniformly continuous, for any $B \in \N$ there exists $\delta>0$ and $K \subset X$, compact with $\mu(K)>1-\epsilon$, so that for all $x,y \in K$ with $d(x,y)<\delta$ and $-B\leq a\leq B $ we have 
\[
\metric(T^ax,T^ay)<\epsilon.
\]
Combining this with \eqref{eq:buddies close} gives the lemma.
\end{proof}

\begin{proof}[Proof of Theorem \ref{thm:getting_constant weak}]
Let $\rho>0$ be given and $f$ be a non-constant $\ltwo(\mu)$ function. Because $T$ is weakly mixing and thus every power of $T$ is ergodic, for all $0<|\ell|<9$ there exists $c_\ell>0$ so that
\[
\int |f-f\circ T^\ell| \intd \mu>c_\ell
\]
holds.
Let $\tilde{c}=\min\{1,c_\ell:0<|\ell|<9\}$ and 
$0<\varepsilon<\frac{\tilde{c}}9 \frac{\rho}{400}.$ By the ergodic theorem (either Birkhoff or Von Neumann), for each  $0<|\ell|<9$ there exists $B_\ell \in \N$ so that for all $B\geq B_\ell$ we have
\begin{equation}
\label{eq:erg}
\mu \left( \left\{ y : \bigg|\frac 1 {B+1} \sum_{k=0}^{B} |f(T^{k}y)-f(T^{\ell+k}y)|-\int|f-f \circ T^{\ell}| \intd\mu\bigg|>\varepsilon\right\} \right) < \varepsilon.
\end{equation}
Denote these sets $A_{\ell,B}$.
Let 
$\tilde{B}=\max\{B_\ell:0<|\ell|<9\}$. 
By Lusin's theorem, 
there exists $K:=K_{\varepsilon,\tilde{B}}\subset X$ compact so that $\mu(K)>1-\varepsilon$

and 
$f\circ T^i$ restricted to $K$ is continuous for all $-9<i<\tilde{B}+9$. Because continuous functions on compact sets are uniformly continuous, 
there exists $\tau_{\varepsilon}$ so that when $x,y \in K_{\varepsilon,B}$ and $x \stackrel{\tau_{\varepsilon}}{=} y$  then 
\begin{equation}
\label{eq:all close} 
|f(T^jx)-f(T^jy)|<\varepsilon.
\end{equation}

We now assume that the assumption of Theorem \ref{thm:getting_constant weak} holds and obtain $\tau$ as in Lemma \ref{lem:same power} for  $\epsilon=\min\{2^{-\tau_{\varepsilon}},\varepsilon\}$ and let $\tau'=\max\{\tau_{\varepsilon},\tau\}.$ We choose $i_0$ as in the statement of Theorem \ref{thm:getting_constant weak} with the $\rho$ we have been given and  $\tau=\tau'$. 

By Lemma \ref{lem:same power}
there exists $0<|\ell|<9$ (depending on $i$) and a set $G_i$ so that $\mu(G_i)\geq \frac 1 {18}\mu(\Gamma_i)-\varepsilon>\frac{1}{20} \rho$ and  
\begin{equation}
\label{eq:specific ell}
(T^{n_i+a} \circ \phi_i)(y) \stackrel{\tau}{=} (T^{n_i+a+\ell})(z)\textup{ for all } -\tilde{B}\leq a \leq \tilde{B} \textup{ and all } z \in G_i
\end{equation}
both hold.

Now let $D$ denote the set of $z$ so that 
\begin{itemize}
\item $z \in K$
\item $\phi_i(z) \in K$
\item $T^{n_i}(z) \in K$
\item $(T^{n_i} \circ \phi_i)(z)\in K$
\item $z \in G_i$
\item $T^{n_i}(z) \in A_{\ell,\tilde{B}}.$ 
\end{itemize}
Observe that 
$\mu(D)=\mu(G_i)-\mu(A_{\ell,\tilde{B}}^c)-4 \mu(K^c)>\frac 1 {40}\rho.$ 
For all $z\in D$ we have by \eqref{eq:erg}, \eqref{eq:all close} and \eqref{eq:specific ell} that

\[
\bigg| \frac{1}{\tilde{B}+1} \sum_{j=0}^{\tilde{B}} |f(T^{n_i+j}(z))-f((T^{n_i+j} \circ \phi_i)(z))| - \int |f-f \circ T^\ell| \intd\mu\bigg| < 2\varepsilon
\]
and so
\[
\frac{1}{\tilde{B}+1} \sum_{j=0}^{\tilde{B}} \Big| f(T^{n_i+j}(z))-f((T^{n_i+j} \circ \phi_i)(z)) \Big| > \tilde{c} - 2\varepsilon
\]
as well.

Because the triangle inequality gives
\begin{multline*}|f(T^{n_i+j}(z))-f((T^{n_i+j} \circ \phi_i)(z))|-
|f(T^j(z))-f((T^j \circ \phi_i)(z))|\leq \\
|f(T^{n_i+j}(z))-f(T^j(z))|+|f((T^{n_i+j} \circ \phi_i)(z))-f((T^j\circ \phi_i)(z))|
\end{multline*}
we can estimate
\begin{align*}
&
\frac{1}{\tilde{B}+1} \sum_{j=0}^{\tilde{B}}\int\limits_{D} \Big| f(T^{n_i+j}(z))-f(T^j(z)) \Big| + \Big| f((T^{n_i+j} \circ \phi_i)(z))-f((T^j \circ \phi_i)(z)) \Big| \intd\mu
\\
\ge{}
&
\mu(D) \tilde{c} - 2\varepsilon - \frac{1}{\tilde{B}+1} \sum_{j=0}^{\tilde{B}}\int\limits_{D} |f(T^j(z))-f((T^j \circ \phi_i)(z))| \intd \mu\geq \frac{1}{80} \rho \tilde{c}
\end{align*}
where the last inequality uses \eqref{eq:all close}. 
Thus 
\begin{equation}
\label{eq:not rigid}
\int\limits_X |f(T^{n_i}x)-f(x)|\geq \frac 1 {160} \rho \tilde{c}
\end{equation}
proving the theorem. 

\end{proof}

In fact the proof gives the following.

\begin{corollary}
\label{cor:end}
For any non-constant $f \in \ltwo(X,\mu)$ and any $\rho > 0$ there exists $\tau':=\tau'(\rho,f) \in \mathbb{N}$ and $\delta:=\delta(\rho,f)>0$ so that if 
\begin{itemize}
\item 
$m \in \mathbb{Z}$,
\item
$\Gamma$ is a Borel set
\item
$\phi:\Gamma \to X$ is $\mu$-measure preserving
\end{itemize} 
satisfy 
\begin{itemize}
\item
$\mu(\Gamma)\geq \rho$ and 
\item
for all $z \in \Gamma$ we have that $(z,\phi(z))$ are $(m,\tau)$-buddies
\end{itemize}
then 
\[
\int_X |f -f\circ T^m| \intd\mu>\delta
\]
holds.
\end{corollary}

\subsection{Buddies from friendly pairs}
\label{subsec:buddies from pairs}

We define friendly pairs and prove that friendly pairs are always buddies. We then show how to produce friendly pairs at bounded scales, 
establishing that if $f\in \ltwo$, non-constant and $\textbf{\textit{r}}$ is a sequence such that $\scl(r(n))$ is at a bounded scale for infinitely many $n$ then $\textbf{\textit{r}}$ is not an $f$-rigidity sequence (see Corollary \ref{cor:bounded mmix}).

\begin{definition} [Friendly pairs]
\label{def:friends} 
Fix $r \in \mathbb{Z}$ and $j \in \N$ with $j$ at a bounded scale. 
We will say that $(y,z)$ in $X \times X$ is an \define{$(r,j)$ friendly pair} if the following properties all hold.
\begin{enumerate}[label={\textbf{\textup{F\arabic*}}.},ref={\textbf{F\arabic*}},leftmargin=5em]
\item \label{cond:start same}
$y(n) \stackrel{j-1}{=} z(n)$ for all $1 \le n \le j-1$
\item \label{cond:T defined 1}
$T^n(y)= (S|Y_j)^n(y)$ for all $1 \le |n| \le |r|+1$ with $sgn(n)=sgn(r)$
\item \label{cond:T defined 2}
$T^n(z) = (S|Y_j)^n(z)$ for all $1 \le |n| \le |r|+1$ with $sgn(h)=sgn(r)$
\item \label{cond:friend miss}
$(S|Y_{j-1})^n(y) \not\in W_j$ for all $1 \le |n| \le |r|+1$ with $sgn(h)=sgn(r)$
\item \label{cond:friend hit}
$(S|Y_{j-1})^h(z) \in W_j$ for a unique $1 \le |h| \le |r|$ with $sgn(h)=sgn(r)$
\end{enumerate}
\end{definition}

The following lemma says that friendly pairs are always buddies.

\begin{lemma}
\label{lem:friend to weak}
If $(y,z)$ is an $(r,j)$ friendly pair

 then $(y,z)$ are $(r,j)$-buddies.
\end{lemma}
\begin{proof}
By conditions \ref{cond:T defined 2} and \ref{cond:friend hit} we have that $T^r(z)=(S|Y_{j-1})^{r+1}(z)$ and by conditions \ref{cond:T defined 1} and \ref{cond:friend miss} we have that $T^r(y)=(S|Y_{j-1})^r(y)$ and $T^{r+1}(y)=(S|Y_{j-1})^{r+1}(y)$. By condition \ref{cond:start same} we have that the first $j-1$ entries of $(S|Y_{j-1})^k(y)$ are the same as the first $j-1$ entries of $(S|Y_{j-1})^k(z)$ for all $k$ and in particular for $k \in \{r,r+1\}$ proving the second condition in the definition of buddies.
This proves the lemma.
\end{proof}

There are some situations where it is relatively straightforward to produce friendly pairs. The first, formalized by the next result, is when $r \in \Z$ is at a bounded scale. The second, given by \cref{prop:small friend}, is when $r$ is small for its scale and $\scl(r) - 1$ is a bounded scale. The second applies in particular when $r$ is at an unbounded scale, but small at that scale. The salient point is that neither of these situations require us to reduce to smaller scales.

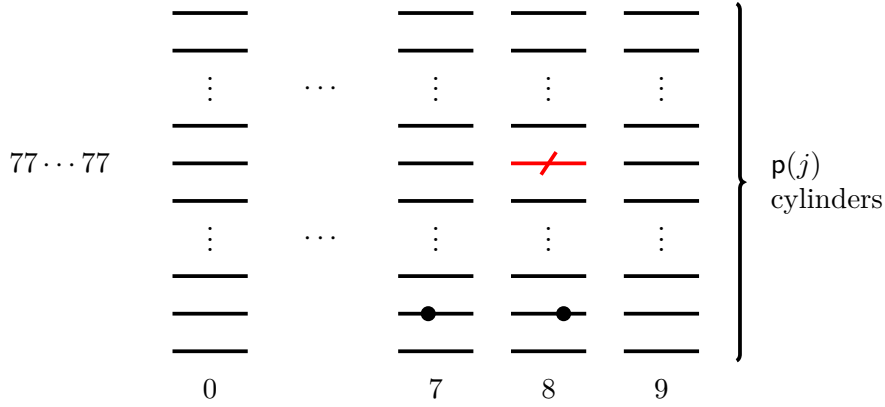
\begin{figure}[t]
\centering
\begin{tikzpicture}[yscale=0.5, line width=0.5mm]

\begin{scope}[shift={(-1.5,0)}]
\node at (0,5) {$77 \cdots 77$};
\end{scope}

\begin{scope}
\draw (0,9) -- (1,9);
\draw (0,8) -- (1,8);
\node at (0.5,7.2) {$\vdots$};
\draw (0,6) -- (1,6);
\draw (0,5) -- (1,5);
\draw (0,4) -- (1,4);
\node at (0.5,3.2) {$\vdots$};
\draw (0,2) -- (1,2);
\draw (0,1) -- (1,1);
\draw (0,0) -- (1,0);
\node at (0.5,-1) {$0$};
\end{scope}

\begin{scope}[shift={(1.5,0)}]
\node at (0.5,7) {$\cdots$};
\node at (0.5,3) {$\cdots$};
\end{scope}

\begin{scope}[shift={(3,0)}]
\draw (0,9) -- (1,9);
\draw (0,8) -- (1,8);
\node at (0.5,7.2) {$\vdots$};
\draw (0,6) -- (1,6);
\draw (0,5) -- (1,5);
\draw (0,4) -- (1,4);
\node at (0.5,3.2) {$\vdots$};
\draw (0,2) -- (1,2);
\draw (0,1) -- (1,1);
\fill (0.4,1) ellipse (0.1 and 0.2);
\draw (0,0) -- (1,0);
\node at (0.5,-1) {$7$};
\end{scope}

\begin{scope}[shift={(4.5,0)}]
\draw (0,9) -- (1,9);
\draw (0,8) -- (1,8);
\node at (0.5,7.2) {$\vdots$};
\draw (0,6) -- (1,6);
\draw[color=red] (0,5) -- (1,5);
\draw[color=red] (0.4,4.7) -- (0.6,5.3);
\draw (0,4) -- (1,4);
\node at (0.5,3.2) {$\vdots$};
\draw (0,2) -- (1,2);
\draw (0,1) -- (1,1);
\fill (0.7,1) ellipse (0.1 and 0.2);
\draw (0,0) -- (1,0);
\node at (0.5,-1) {$8$};
\end{scope}

\begin{scope}[shift={(6,0)}]
\draw (0,9) -- (1,9);
\draw (0,8) -- (1,8);
\node at (0.5,7.2) {$\vdots$};
\draw (0,6) -- (1,6);
\draw (0,5) -- (1,5);
\draw (0,4) -- (1,4);
\node at (0.5,3.2) {$\vdots$};
\draw (0,2) -- (1,2);
\draw (0,1) -- (1,1);
\draw (0,0) -- (1,0);
\node at (0.5,-1) {$9$};
\end{scope}

\begin{scope}[shift={(7.5,0)}]
\draw [decorate, decoration = {brace}] (0,9.25) --  (0,-0.25);
\node[align=left] at (1.2,4.5) {$\p(j)$\\cylinders};
\end{scope}

\end{tikzpicture}

\caption{The picture of $Y_{j-1}$ at the bounded scale $j$. The marked interval is $W_j$. The marked points form a $(\p(j), j)$ friendly pair if both points are outside $\bad(j,\p(j))$.}
\label{fig:acquainted_pair}

\end{figure}

\begin{theorem}
\label{thm:bounded wfriends}
For all $r \in \Z$ with $\scl(r)$ bounded there is a set $B \subset X$ with the following properties.

\begin{itemize}
\item
$\mu(B) \ge 10^{-5}$
\item
There is $\Phi : B \setminus \bad(\scl(r),8r) \to X$ a measure-preserving map that is a bijection onto its image such that $(\Phi(z), z)$ is are $(r,\scl(r))$-friends for all $z \in B \setminus \bad(\scl(r),8r)$.
\end{itemize}
\end{theorem}

\begin{proof}
Fix $r \in \Z$ with $\scl(r)$ bounded.
We prove the case $r < 0$.
The case $r > 0$ is similar and therefore omitted.
Note that
\begin{equation}
\label{eqn:wee bound}
\p(\scl(r)+1)/2\ge|r| \ge \p(\scl(r)) / 2
\end{equation}
holds from the definition of the expansion. Define
\begin{gather}
\label{eqn:Bset}
B = \Big\{ x \in X : x(\scl(r)) = 8 , x(\scl(r)+1)\in [\tfrac{(c(\scl(r)+1)+1)}{5},\tfrac{(c(\scl(r)+1)+1)}{2}]  \\
\textup{ and } (S|Y_{\scl(r)-1})^i(x) \in W_{\scl(r)} \textup{ for some } -|r| < i \le -1 \Big\}
\end{gather}
and note that $\nu(B) \ge 10^{-3}$ by independence so $\mu(B) \ge 10^{-3}$ by definition of $\mu$.

Given $y \in B \setminus \bad(\scl(r),8r)$ define $\Phi(y)$ by $(\Phi(y))(\scl(r)) = 6$ and $(\Phi(y))(n) = y(n)$ otherwise.
Fix $z \in B \setminus \bad(\scl(r),8r)$.
The pair $(\Phi(z),z)$ satisfies \ref{cond:start same}.
Observe that $(S|Y_{\scl(r)-1})^h(z) \in W_{\scl(r)}$ for exactly one $r< h\leq-1$ because of \eqref{eqn:wee bound} giving \ref{cond:friend hit}.
We also have \ref{cond:T defined 2} because 
\[
(S|Y_{\scl(r)})^h(z) \notin \bigcup_{i>\scl(r)} W_i
\]
for all $-r \le h \le -1$ follows upon combining $z \notin \bad(\scl(r),8r)$ with the condition on $z(\scl(r) + 1)$.
As $(\Phi(z))(\scl(r)) = 6$ we have $(S|Y_{\scl(r)-1})^h(\Phi(z)) \notin W_{\scl(r)}$ for all $-r-1 \le h \le -1$.
That gives \ref{cond:friend miss}.
Next note that $\Phi(z) = T^n(z)$ for some $4r<n<0$ which implies $\Phi(z) \in X$.
As $z \notin \bad(\scl(r),8r)$, we have
\[
(S|Y_{\scl(r)-1})^h(\Phi(z)) \notin  \bigcup_{a(k)>\scl(r)}W_{a(k)}
\]
which, together with the fact that $(\Phi(z))(\scl(r)) = 6$ implies
\[
(S|Y_{\scl(r)-1})^h(\Phi(z)) \notin \bigcup_{j>\scl(r)} W_j
\]
for all $r \le h \le 0$, gives \ref{cond:T defined 1}.
Thus $(\Phi(z),z)$ is an $(r,\scl(r))$ friendly pair.
\end{proof}

\begin{corollary}
\label{cor:bounded mmix}
If $\textbf{r}$ is a sequence in $\Z$ with $|r(i)| \to \infty$ and $r(i)$ at a bounded scale infinitely often then $\textbf{r}$ is not a rigidity sequence for any non-constant $\lp^2(X,\mu)$ function. 
\end{corollary}

\begin{proof}
 Pass to a subsequence with the property that $r(i)$ is at a bounded scale for all $i \in \N$. Let $B_i$ and $\Phi_i$ result from applying \cref{thm:bounded wfriends} to each $r(i)$.
By \cref{lem:friend to weak} and \cref{thm:getting_constant weak} it suffices to show
\[
\mu\Big( B_i \cap \bad(\scl(r(i)),8r(i)) \Big) < \frac{8}{9} \cdot \frac{1}{10^5}
\]
for all large enough $i$.
Fix $i$ and write $r$ for $r(i)$ and $B$ for $B_i$.
For convenience of exposition we assume $r < 0$. Because of the fast decay of $j \mapsto \mu(W_{a(j)})$, it suffices to show that if $a(k)$ is the smallest unbounded scale strictly greater than $\scl(r)$ we have 
\[
\mu(\{x\in B:(S|Y_{\scl(r)})^i(x)\in W_{a(k)} \text{ for some }-8|r|<i<8|r|\})<\frac 3 4 \mu(B).
\]
Observe that
\[
\{x\in X:(S|Y_{\scl(r)})^i(x)\in W_{a(k)} \text{ for some }-8|r|<i<8|r|\}
\]
is contained in $\{x \in X : x(a(k))\in [0,9] \cup [d(k)-8,2d(k)+1 ] \}$.
From the definition of $B$, \eqref{eqn:Bset}, if $\scl(r)=a(k)-1$  then
\[
\{x\in B:(S|Y_{\scl(r)})^i(x)\notin W_{\scl(r)+1}=W_{a(k)} \text{ for all }-8|r|<i<8|r| \}
\]
and otherwise
\[
\nu(B \setminus \{x\in X:(S|Y_{\scl(r)})^i(x)\in W_{a(k)} \text{ for some } -8|r|<i<8|r|\})\geq \frac{d(k)}{2d(k)+2}\nu(B)
\]
and
analogously to the proof of Lemma \ref{lem:indep} we have our estimate.
\end{proof}

We conclude this section with a second situation in which friends can be produced directly: when the time $r$ is relatively small for its scale and $\scl(r) - 1$ is a bounded scale.

\begin{proposition}
\label{prop:small friend}
Fix $r \in \Z$.
If $\scl(r) - 1$ is bounded and $|r|\leq \frac{3}{4} \p(\scl(r))$ then there exists a Borel set $B$ with $\mu(B)>10^{-5}$ and $\Phi:B \setminus \bad(\scl(r),2r) \to X$ measure preserving bijection onto its image, so that for all $x \in B$, $(x,\Phi(x))$ are $(r,\scl(r)-1)$ friends.
\end{proposition}
\begin{proof}
By Theorem \ref{thm:bounded wfriends} may assume $r$ is at an unbounded scale.
Thus $\scl(r) \in \{a(\ell)\}$.
For readability, in this proof we write $b$ for $\scl(r)$.
It suffices to construct a set of $(x,x')$ that form an $(r,b-1)$-friendly pair.
We present this in the case that $r>0$.
The case $r<0$ is similar.
The pairs are formed by $x$ so that $x(b-1)=9$, 
\begin{equation}
\label{eq:index good}
x(b) < d(b)-2, \text{ and } x(b+1)\notin \{6, 7,8\}.
\end{equation} 
and $x'$ so that  
\begin{equation}
\label{eq:index same}
x'(b-1) = 6 \text{ and } x'(\ell)=x(\ell) \text{ for all }\ell \neq b-1.
\end{equation} 
Observe that $x \stackrel{b-2}{=} y$ gives \ref{cond:start same}.
Also, from \eqref{eq:index good} and \eqref{eq:index same}, $(S|Y_{b-1})^k(x') \in W_{b-1}$ for some $0<k<i$ and this is not true for $x$. This gives \ref{cond:friend hit} and \ref{cond:friend miss}. 
Further observe that $(S|Y_{b-1})^i \{x,x'\} \cap W_\ell = \emptyset$ for all $0<k<i+1$ and $\ell \geq b$ with $\ell \notin \{a(j)\}$.
From these pairs we remove any pair $x,x'$ so that at least one is in  $\bad(b,i)$ gives \ref{cond:T defined 1} and \ref{cond:T defined 2}.
Because $x'=T^n(x)$ for $-i<n<0$, we have that $x' \in \bad(b,i)$ implies $x \in \bad(b,2i)$.
Thus $(x,x')$ are $(i,b-1)$-friends whenever $x \notin \bad(b,2i)$.
Defining $\Phi$ to change the $b-1$ digit from $9$ to $6$ we obtain the proposition. 

\end{proof}

\subsection{The reduction argument}
\label{subsec:reduction summary}

We have established via \cref{cor:bounded mmix} that any non-constant $f \in \ltwo(X,\mu)$ cannot be rigid along $\textbf{\textit{r}}$ if, infinitely often, the time $r(n)$ is at a bounded scale.
It remains (after passing to a subsequence if necessary) to consider the complementary case: that $r(n)$ is at an unbounded scale for all $n \in \N$.
Suppose therefore that every term of the sequence $\textbf{\textit{r}}$ has tower expansion
\[
r(n) = \sum_{j=1}^{\scl(r(n))} e_n(j) \p(j)
\]
with $\scl(r(n))$ unbounded for all $n \in \N$. This is the most technical part of the argument.

We are unable to produce $(r(n), \scl(r(n))-1)$ friends or buddies directly when $r(n)$ is at an unbounded scale as each $W_{a(j)}$ consists of a large number of cylinders.
Our strategy in this case is to fix in advance an unbounded scale $\tau$ that is much smaller than $\scl(r(n))$ and to understand $T^{r(n)}$ at scale $\tau$ instead of the more natural scale of $\scl(r(n))$.
To begin to understand $T^{r(n)}$ at lower scales we use the following observation: in the picture at scale $\scl(r(n))$ there is a large set of points $A^{(1)}_1$ where $T^{\p(\scl(r(n)))} \approx \id$ and a large proportion of points $A^{(1)}_2$ where $T^{\p(\scl(r(n))) - \q(\scl(r(n)))} \approx \id$.
We then have $T^{r(n)} \approx T^{\lft(r(n))}$ and $T^{r(n)} \approx T^{\rht(r(n))}$ on sets $A^{(1)}_1$ and $A^{(1)}_2$ respectively where
\begin{align*}
\lft(m) &= m - \lst(m) \p(\scl(m)) \\
\rht(m) &= m - \lst(m) \p(\scl(m)) + \lst(m) \q(\scl(m))
\end{align*}
are the reduction operations that we will apply iteratively.

A consequence of the above is that we need to keep track not only of the times we obtain by applying $\lft$ and $\rht$ but also paired sets of points where the approximations hold.
As we repeatedly reduce (i.e. compose with $\lft$ and $\rht$ again and again) there is a proliferation of such pairs of times and sets.
We need to analyze how they are related, whether their associated times are at bounded or unbounded scales, and various other features.
We prepare to undertake this analysis in \cref{subsec:reduction once} by carefully characterizing (cf.\ \cref{cor:pos reduc conseq} and \cref{cor:neg reduc conseq}) which points are well-behaved under a single application of $\lft$, which are well-behaved under a single application of $\rht$, and which are neither.

The analysis of repeated reduction is the content of \cref{prop:reduction}.
It records in \ref{cond:setup} through \ref{cond:paired cyl} what happens as we reduce a time $m$ at an unbounded scale $\scl(m)$ to a smaller unbounded scale $\lambda^I(\scl(m))$ above our target unbounded scale $\tau$.
\ref{cond:setup} and \ref{cond:partition} records the data we keep track of: pairs of points and times $(A^{(i)}, n^{(i)}_j)$ that result from composing $i$ times in various ways the maps $\lft$ and $\rht$ and a repository $A^{(i)}_0$ of ``left over'' points.
\ref{cond:agree} records how times for different values of $i$ are related.

We see from \ref{cond:Azero friends} that a positive proportion of the leftover points $A^{(i)}_0$ belong to friendly pairs.
The sets $A^{(1)}_0,A^{(2)}_0,\dots,A^{(I)}_0$ are nested, and it may be that $\mu(A^{(i)}_0)$ is large enough for \ref{cond:Azero friends} to be compatible with \cref{cor:end}.

This leaves the case that $A^{(I)}_0$ is not particularly large.
When this happens we need to worry about the possibility that some scales $\scl(n^{(i)}_j)$ are smaller than $\tau$.
It is certainly possible (consider $m = \p(\scl(m))+1$ for example) that scales drop precipitously when $\lft$ or $\rht$ is applied.
However (cf.\ \cref{lem: reduction not both small}) it turns out that $\lft(m)$ and $\rht(m)$ cannot both behave badly in this sense.
To make use of this along the reduction procedure we pair up our times $n^{(i)}_1,\dots,n^{(i)}_{k(i)}$ for each $1 \le i \le I$ in such a way that at most one of them is at a scale coarser than $\tau$.
The properties we need of this pairing are given by \ref{cond:differ} and \ref{cond:paired cyl}.

All of this puts us in the following position: we have for each large enough $n$ a description of what happens as each $r(n)$ is reduced $I(n)$ times to the target unbounded scale $\tau$.
If for infinitely many $n$ the set $A^{I(n)}_0$ is large enough to apply \cref{cor:end} then $\textit{\textbf{r}}$ is not a rigidity sequence for any non-constant $f \in \ltwo(X,\mu)$.
If that does not happen, we are left with times $n^{(I(n))}_1,\dots,n^{(I(n))}_{k(I(n))}$ at least half of which are at scale $\tau$ and behave like $r(n)$ with respect to $T$.
As $\tau$ is relatively small we can scale each of these times by a relatively small amount $d$ so that half of them have expansions at the bounded scale $\tau+1$.
We will be able to conclude the whole argument if it is the case that the times
\[
d n^{(I(n))}_1,d n^{(I(n))}_2,\dots,d n^{(I(n))}_{k(I(n))}
\]
are what result if the whole procedure above is instead applied to the scaled sequence $d \textit{\textbf{r}}$.
Thus, when reducing we want also to keep track of how scaling affects everything.
What we need for that is recorded in \ref{cond:paired numbers}.

\subsection{Reducing at a single scale}
\label{subsec:reduction once}

\begin{lemma}
\label{lem:reduction positive}
Fix $r \in \N$ with $\scl(r)$ an unbounded scale $a(b)$ and fix $x \in X$.
\begin{enumerate}[label={\textbf{\textup{\ref{lem:reduction positive}.\arabic*}}.},ref={\textbf{\ref{lem:reduction positive}.\arabic*}},leftmargin=3em]
\item
\label{lem:pos red S left}
If $1 \le x(\scl(r)) \le d(b) - \lst(r)$ then $(S|Y_{\scl(r)})^r(x)$ and $(S|Y_{\scl(r)})^{\lft(r)}(x)$ agree in their first $\scl(r)-1$ coordinates.
\item
\label{lem:pos red S right}
If $d(b) + 2 \le x(\scl(r)) \le c(b) - \lst(r)$ then $(S|Y_{\scl(r)})^r(x)$ and $(S|Y_{\scl(r)})^{\rht(r)}(x)$ agree in their first $\scl(r)-1$ coordinates.
\item
\label{lem:pos red S bdd}
If $\lst(r) \ge 11$ and $c(b) - \lst(r) + 10 \le x(\scl(r)) \le c(b)$ then $((S|Y_{\scl(r)})^i (x))(\scl(r)) = 0$ for some $0 \le i < r$. If $x(\scl(r)+1) = 7$ as well there is $0 \le i \le r$ with $(S|Y_{\scl(r)})^i(x) \in W_{\scl(r)+1}$.
\end{enumerate}
\end{lemma}
\begin{proof}
For \ref{lem:pos red S left} the effect of $(S|Y_{\scl(r)})^{r - \lft(r)}=(S|_{Y_{\scl(r)}})^{\lst(r)\p(r)}$ on any such $x$ is to increase the $\scl(r)$ coordinate by the amount $\lst(r)$ to at most $d(b)$.
When we apply $(S|Y_{\scl(r)})^{\lft(r)}$ to both $x$ and $((S|Y_{\scl(r)})^{r - \lft(r)})(x)$ the resulting points continue to agree in their first $\scl(r) - 1$ coordinates by \cref{lem:need it later}. Indeed, when $\lft(r)$ is negative the condition $1 \le x(\scl(r))$ together with $|\lft(r)| \le \p(\scl(r))$ avoids $((S|Y_{\scl(r)})^{\lft(r)})(x)$ having a negative $\scl(r)$ coordinate, and when $\lft(r)$ is positive $|\lft(r)| \le \p(\scl(r)) - \q(\scl(r))$ prevents $(S|Y_{\scl(r)})^r(x)$ from having a $\scl(r)$ coordinate strictly larger than $ d(b) + 1$.

For \ref{lem:pos red S right} the effect of $(S|Y_{\scl(r)})^{r - \rht(r)}$ on any such $x$ is to increase the $\scl(r)$ coordinate by the amount $\lst(r)$ to at most $c(b)$.
When we apply $(S|Y_{\scl(r)})^{\rht(r)}$ to both $x$ and $((S|Y_{\scl(r)})^{r - \rht(r)})(x)$ the resulting points continue to agree in their first $\scl(r) - 1$ coordinates by \cref{lem:need it later}. Indeed, when $\rht(r)$ is negative the condition $d(b) + 2 \le x(\scl(r))$ together with $|\rht(r)| \le \p(\scl(r)) - \q(\scl(r))$ avoids $((S|Y_{\scl(r)})^{\rht(r)})(x)$ having a $\scl(r)$ coordinate less than or equal to $d(b)$, and when $\rht(r)$ is positive $|\rht(r)| \le \p(\scl(r)) - \q(\scl(r))$ prevents $(S|Y_{\scl(r)})^r(x)$ from going too far up the first tower i.e.\ having $\scl(r)$ coordinate greater than 0 or have $\scl(r)$ coordinate equal to 0 and being in a level that would be in $W_{\scl(r)}$ on the right side towers.

For \ref{lem:pos red S bdd} note that $|r - \lst(r) \p(\scl(r))| \le \p(\scl(r))$ by \cref{cor:obvioius_thing} because
\[
\scl(r - \lst(r) \p(r) ) \le \scl(r) - 1
\]
so that for any $x$ that satisfies $c(b) - \lst(r) + 10 \le x(\scl(r)) \le c(b)$ there is $0 \le j \le r$ such that the $\scl(r)$ coordinate of $(S|Y_{\scl(r)})^j(x)$ is $7$.
In particular, there must be some $0 \le i \le r$ for which $(S|Y_{\scl(r)})^i(x)$ belongs to $[7^{\scl(r)}]$ and if additionally $x(\scl(r) + 1) = 7$ then we will have $(S|Y_{\scl(r)})^j(x)$ in $W_{\scl(r)+1}$. Indeed, $(S|Y_{J(r)})^\ell (x)(J(r))$ going from $c(b)$ to 0 makes $(S|Y_{J(r)})^\ell (x)(J(r)+1)$ go from $7$ to $8$.

\end{proof}

\begin{corollary}
\label{cor:pos reduc conseq}
Fix $r \in \N$ with $\scl(r)$ an unbounded scale $a(b)$ and fix $x \in X$.
\begin{enumerate}[label={\textbf{\textup{\ref{cor:pos reduc conseq}.\arabic*}}.},ref={\textbf{\ref{cor:pos reduc conseq}.\arabic*}},leftmargin=3em]
\item
\label{cor:pos red T left}
If $9 \le x(\scl(r)) \le d(b)-\lst(r)$ and $x \not\in \bad(\scl(r),r)$ then
\[
(S|Y_{\scl(r)})^r(x)=T^r(x)
\qquad
(S|Y_{\scl(r)})^{\lft(r)}(x) = T^{\lft(r)}(x)
\]
and these two points agree on the first $\scl(r)-1$ coordinates.
\item
\label{cor:pos red T right}
If $d(b) + 2 \le x(\scl(r)) \le c(b) - \lst(r)$ and $x \not\in \bad(\scl(r),r)$ then
\[
(S|Y_{\scl(r)})^r(x) = T^r(x)
\qquad
(S|Y_{\scl(r)})^{\rht(r)}(x) = T^{\rht(r)}(x)
\]
and these two points agree in their first $\scl(r)-1$ coordinates.
\item
\label{cor:pos red T friends}
If $\lst(r) \geq 11$ and $c(b) - \lst(r) + 10 \le x(\scl(r)) \le c(b)$ and $x(\scl(r)+1) = 7$ and $x \not\in \bad(\scl(r),r+1)$ then $(y,x)$ are $(r,\scl(r)+1)$-friends whenever $y$ satisfies $x(k) = y(k)$ for all $k \le \scl(r)$ and $y(\scl(r)+1) = 3$ and $y \not\in \bad(\scl(r),r)$.
\end{enumerate}
\end{corollary}
\begin{proof}
For \ref{cor:pos red T left} the condition $x(\scl(r)) \ge 9$ guarantees the $\scl(r)$ coordinate of $(S|Y_{\scl(r)})^i(x)$ is not 7 for all $-|\lft(r)| \le i\le r$ by \cref{lem:need it later}.
Together with the assumption that $x$ does not belong to $\bad(\scl(r),r)$ it follows that $(S|Y_{\scl(r)})^i(x) \notin W_k$ for all $-|\lft(r)| \le i \le r$ and all $k > \scl(r)$.
Thus $(S|Y_{\scl(r)})^i(x) = T^i(x)$ for all $-|\lft(r)| \le i\le r$.
As $x$ satisfies in particular the hypothesis of \ref{lem:pos red S left} we get the desired agreement in the first $\scl(r)-1$ coordinates.

To prove \ref{cor:pos red T right} note there is no possibility of $(S|Y_{\scl(r)})^i(x)$ belonging to some $W_k$ with $k > \scl(r)$ for some $-|\rht(r)| \le i \le r$ and we have the desired conclusion by appealing to \ref{lem:pos red S right}.

For \ref{cor:pos red T friends} we need to verify the properties in \cref{def:friends}.
Condition \ref{cond:start same} is one of our assumptions.
To prove \ref{cond:T defined 1} and \ref{cond:friend miss} combine $y \not\in \bad(\scl(r),r)$ with $y(\scl(r) + 1) = 3$ to get $(S|Y_{\scl(r)})^i(y) \not\in W_k$ for all $k \ge \scl(r)+1$ and all $1 \le i \le r$. For \ref{cond:T defined 2} and \ref{cond:friend hit} combine $x \not\in \bad(\scl(r),r+1)$ with the facts that there is a unique $1 \le i \le r$ for which $(S|Y_{\scl(r)})^i(x)$ belongs to $[7^{\scl(r)}]$ and that $x(\scl(r) + 1) = 7$ allows \ref{lem:pos red S bdd} to guarantee there is $1 \le i < r$ with $(S|Y_{\scl(r)})^i(x)$ in $W_{\scl(r)+1}$.
\end{proof}

We have additionally the following reduction statements when the time $r$ is negative. The proofs are entirely analagous.

\begin{lemma}
\label{lem:reduction negative}
Fix $r \in -\N$ with $\scl(r)$ an unbounded scale $a(b)$ and fix $x \in X$.
\begin{enumerate}[label={\textbf{\textup{\ref{lem:reduction negative}.\arabic*}}.},ref={\textbf{\ref{lem:reduction negative}.\arabic*}},leftmargin=3em]
\item
\label{lem:neg red S left}
If $-\lst(r)+1 \le x(\scl(r)) \le d(b) -1$ then $(S|Y_{\scl(r)})^r(x)$ and $(S|Y_{\scl(r)})^{\lft(r)}(x)$ agree in their first $\scl(r) - 1$ coordinates.
\item
\label{lem:neg red S right}
If $d(b) + 2 - \lst(r) \le x(\scl(r)) \le c(b)-2$ then $(S|Y_{\scl(r)})^r(x)$ and $(S|Y_{\scl(r)})^{\rht(r)}(x)$ agree in their first $\scl(r)-1$ coordinates.
\item
\label{lem: neg red S bdd}
If $\lst(r) \leq -11$ and $10 \le x(\scl(r)) \le -\lst(r) + 1$ then $((S|Y_{\scl(r)})^i(x))(\scl(r)) = 0$ for some $0 \ge i \ge r$.
If $x(\scl(r) + 1) = 8$ as well there is $0 \ge i \ge r$ with $(S|Y_{\scl(r)})^i(x) \in W_{\scl(r) + 1}$.
\end{enumerate}
\end{lemma}

\begin{corollary}
\label{cor:neg reduc conseq}
Fix $r \in -\N$ with $\scl(r)$ an unbounded scale $a(b)$ and fix $x \in X$.
\begin{enumerate}
[label={\textbf{\textup{\ref{cor:neg reduc conseq}.\arabic*}}.},ref={\textbf{\ref{cor:neg reduc conseq}.\arabic*}},leftmargin=3em]
\item
\label{cor:neg red T left}
If $- \lst(r) + 10 \le x(\scl(r)) \le d(b)-1$ and $x \notin \bad(\scl(r),r)$ then
\[
(S|Y_{\scl(r)})^r(x) = T^r(x)
\qquad
(S|Y_{\scl(r)})^{\lft(r)}(x) = T^{\lft(r)}(x)
\]
and these two points agree in the first $\scl(r)-1$ coordinates.
\item
\label{cor:neg red T right}
If $d(b) + 2 - \lst(r) \le x(\scl(r)) \le c(b)-2$ and $x \notin \bad(\scl(r),r)$ then
\[
(S|Y_{\scl(r)})^r(x)=T^r(x)
\qquad
(S|Y_{\scl(r)})^{\rht(r)}(x)=T^{\rht(r)}(x)
\]
and these two points agree in their first $\scl(r)-1$ coordinates.
\item
\label{conc:neg reduc friends}
If $\lst(r) \leq -11$ and $10 \le x(\scl(r)) \le -\lst(r) + 1$ and $x(\scl(r)+1) = 8$ and $x \notin \bad(\scl(r),r+1)$ then $(y,x)$ are $(r,\scl(r) + 1)$-friends whenever $y$ satisfies $x(k) = y(k)$ for all $k \le \scl(r)$ and $y(\scl(r) + 1) = 3$ and $y \notin \bad(\scl(r),r)$.
\end{enumerate}
\end{corollary}

\begin{corollary}
\label{cor:small e(J)}
Fix $r \in \Z$ with $\scl(r)$ an unbounded scale $a(b)$ and $|\lst(r)| < 10$.
There exists $B \in X$ measurable and $\Phi:B \to X$ bijective and measurable on its image so that $(x,\phi(x))$ are $(r,\scl(r)+1)$-friends and 
$$\mu(B)\geq \frac{1}{2} \frac {1}{20} \frac{1}{20} \frac{1}{d(b)}.$$
\end{corollary}
\begin{proof}
We prove the case $r>0$. The case $r<0$ is analogous.
Let $J = \scl(r)$ and let
\[
B=\{y\in \Omega :(S|Y_{J(r)})^i(x) \in W_{\scl(r) + 1} \text{ for some } 0 \leq i<r \textup{ and } x(\scl(r) + 2) = 5 \} \cap X
\]
which is the intersection with $X$ of a union of cylinders of length $\scl(r)$.
Let $\phi:B \to X$ by $\phi(x)_j=x_j$ for all $j \neq J+1$ and $\phi(x)_{J+1}=4$. As in the proof of Corollary \ref{cor:pos reduc conseq} and \ref{cor:neg reduc conseq} (3) we have that $x,\phi(x)$ are $(r,J+1)$-friends so long as $x\notin V(b,r)$.  
For our measure estimate, we have that 
$$\mu(\{x:T^ix \in S^{-1}W_{J+1} \text{ for some }0\leq i<r\})>\frac r {t(J+1)}>\frac 1 {20d(J)}.$$ By our independence result, Lemma \ref{lem:indep} applied with $k=1$,  and the fact that there are 10 symbols we have that adding the condition that $x(J+2)=5$ reduces the measure by less than a factor of $2 \cdot 10$. Lastly, the condition not being in $V(b,r)$  
removes less than 
$$2 \cdot r   \sum_{a(j)>J(r)}\frac{Q(j)}{P(j)}< 2 \cdot 10 p(J(r)) \sum_{a(j)>J(r)}\frac{Q(j)}{P(j)}<\frac 1 2 \cdot \frac 1 {20}\frac 1 {20} \frac 1 {d(J)}$$
giving the claim.
\end{proof}

By similar estimates to what we have done earlier \ref{cor:pos red T friends}, \ref{conc:neg reduc friends} and Corollary \ref{cor:small e(J)} give the following as well.

\begin{lemma}
\label{lem:friends}
Fix $r \in \N$ with $\scl(r)$ unbounded.
With
\[
A_0 = \{ x \in X : x \textup{ does not satisfy first assumption of \ref{cor:pos red T left} or \ref{cor:pos red T right}} \}
\]
there is $B\subset A_0$, and $\Phi:B \to X$, a measure preserving injection so that $\mu(B)>\mu(A_0)10^{-6}$ and $(x,\Phi(x))$ are $(r,\scl(r)+1)$-friends.

The same is true with
\[
A_0 = \{ x \in X : x \textup{ does not satisfy first assumption of \ref{cor:neg red T left} or \ref{cor:neg red T right}} \}
\]
instead.
\end{lemma}

\subsection{Key result for unbounded scales}
\label{subsec:reduction main}

In this section we state and prove the key result needed for the proof of \cref{thm:mild mixing}. Fix an unbounded scale $\tau \in \N$.
Given $r \in \Z$ with $\scl(r) > \tau$ the proposition keeps track -- using \cref{cor:pos reduc conseq} and \cref{cor:neg reduc conseq} -- of the behavior of points under first $T^{\lft(r)}$ and $T^{\rht(r)}$ and then under $T^{\lft(\lft(r))}, T^{\rht(\lft(r))}, T^{\lft(\rht(r))}$ and $T^{\rht(\rht(r))}$ and so on until we have applied compositions of $\lft$ and $\rht$ enough times to arrive at the scale $\tau$. There are various intricacies to keep track of.
For example, it may be the case that $\lft(\rht(\lft(r)))$ is at a bounded scale, it may be that $\lst(\lft(\lft(r)))$ is either very small or very large, and it may be that some compositions of $\lft$ and $\rht$ reduce $r$ to a scale smaller than $\tau$.
We begin with a preparatory lemma.

\begin{lemma}
\label{lem:for pairing}
If $0<|k|\leq 2d(r)+3$ then either $\scl(k\Q(r))=a(r-1)$ or $k\Q(r)$ is at a bounded scale. 
Moreover, either
\[
\scl \Big(\lft(k\Q(r))\Big)
=
\scl \Big(\rht(k\Q(r))\Big)
=
a(r-2)
\]
or at least one of $\scl(\lft(k\Q(r)))$, $\scl(\rht(k\Q(r)))$ is at a bounded scale.
Moreover, in the first case, both $\scl(\lft(k\Q(r)))$ and $\scl(\rht(k\Q(r)))$ have the form $k'\Q(r-1)$ for $|k'|\leq 2d(r-2)+3$. 

If $\scl(b)=a(r-\ell)$ for $\ell \geq 2$ and $\scl(k \Q(r)+b)<a(r)$ then $\scl(k\Q(r)+b)=a(r-1)$ or $k \Q(r)+b$ is at a bounded scale. Moreover
\[
\scl(\lft(k\Q(r)+b)) = \scl(\rht(k\Q(r)+b)) = a(r-2)
\]
or at least one of $\scl(\lft(k\Q(r)+b))$, $\scl(\rht(k\Q(r)+b))$ is at a bounded scale greater than $a(r-2)$.
Moreover, in the first case,  both $\scl(\lft(k\Q(r)+b))$ and $\scl(\rht(k\Q(r)+b))$ have the form $k\Q(r-1)+b$ for $|k|<2d(r-1)+2=c(a(r-1))+1$.
\end{lemma}
\begin{proof}
The first claim is because $\P(r-1)<\Q(r)< \p(a(r-1)+1)$ and $(2d(r)+2)\p(a(r-1)+1)$ is much smaller than $\p(a(r)-1)$. 
In particular, $k\Q(r)$ is at a scale between $a(r-1)$ and $a(r)-1$ and any such scale other than $a(r-1)$ is a bounded scale. 

Now we consider the reductions $\lft$ and $\rht$.
Recall, $k\Q(r)=k(2\P(r-1)-\Q(r-1))$, 
so if it is not at a bounded scale, $\scl(k\Q(r))=a(r-1)$ and by the greedy algorithm, $\lst(k\q(r))=2k$ because $(2d(r)+3)\Q(r-1)$ is much smaller than $\P(r-1)$.
If $\scl(\lft(k\Q(r)))$ is at an unbounded scale, similarly to above $\scl(\lft(k\Q(r)) = \scl(k\Q(r-1)) = a(r-2)$.
Because $\rht(k\Q(r))=2k\Q(r-1)+\lft(k\Q(r))=3k\Q(r-1)$ we have that it is either at a bounded scale or at $a(r-2)$. 
Now for the bound on $k'$ if it was not true $\lft(\scl(\Q(r)))$ would be at a bounded scale.

The second paragraph follows similarly to the above. In particular, for the ``moreover'' parts of the second paragraph, $b$ is small enough that adding $b$ does not affect the greedy algorithm.
\end{proof}

\begin{proposition}
\label{prop:reduction}
Fix $\tau \in \N$ a target scale greater than or equal to $a(10)$.
Fix $m \in \N$ with $\scl(m)$ unbounded and $\scl(m) > \tau$.
Fix $I \in \N$ with $\lambda^I(\scl(m)) \ge \tau$.
For every $1 \le i \le I$ there is a tuple
\begin{equation}
\label{eq:tuple}
\left( A_0^{(i)}, A_1^{(i)},n_1^{(i)} , \dots, A_{k(i)}^{(i)},n^{(i)}_{k(i)} \right)
\end{equation}
with all the following properties.
\begin{enumerate}[label={\textbf{\textup{E\arabic*}}.},ref={\textbf{E\arabic*}},leftmargin=5em]
\item
\label{cond:setup}
$A_j^{(i)} \subset X$ for all $0 \le j \le k(i)$ and $n^{(i)}_j \in \Z$ for all $1 \le j \le k(i)$ and $k(i) \in 2\N \cup \{0\}$.
\item
\label{cond:partition}
The sets $A^{(i)}_0,\dots,A^{(i)}_{k(i)}$ form a measurable partition of $X$ in which each $A^{(i)}_j$ is the intersection with $X$ of a finite union of cylinders, each of which is defined by fixing values at unbounded scales between $\scl(m)$ and $\lambda^{i-1}(\scl(m))$.
\item
\label{cond:agree}
Every $1 \le j \le k(i)$ has an ancestor $1 \le \alpha(j) \le k(i-1)$.
If $j \ne 0$ and $x \in A_j^{(i)}$ and
\begin{equation}
\label{eqn:bad ancestors}
x \notin \bigcup_{s=0}^i \bad \Big( \scl \Big( n^{(i-s)}_{\alpha^s(j)} \Big), n^{(i-s)}_{\alpha^s(j)} \Big)
\end{equation}
then 
\begin{equation}\label{eq:same coords}(T^m(x))(\ell) = (T^{n_j^{(i)}}(x))(\ell) \text{ for all }\ell\leq \lambda^{i-1}(\scl(m))-1.
\end{equation}
\item
\label{cond:Azero friends}
There exists $F\subset A^{(i)}_0$ with $\mu(F)>10^{-9}\mu(A^{(i)}_0) $ and $\psi:F \to X$ measure preserving and $j \ge \tau$ so that $(x,\psi(x))$ are $(m,j)$ buddies and in particular $(m,\tau)$ buddies.
\item
\label{cond:differ}
The numbers $1,\dots,k(i)$ are paired so that each $1 \le j \le k(i)$ is in exactly one pair. Moreover, whenever $j$ and $h$ are paired then at least one of $n_j^{(i)}$ and $n_h^{(i)}$ is at scale $\lambda^i(\scl(m))$.
\item
\label{cond:paired numbers}
If $j$ and $h$ are paired  with $\lambda^i(\scl(m)) = \scl(n_j^{(i)}) \ne \scl(n_h^{(i)})$ then $n_j^{(i)}=k\q(\lambda^{i-1}(\scl(m)))+n_h^{(i)}$ with $2|k| \le c(\lambda^{i-1}\scl(m)) + 1$.
\item
\label{cond:paired cyl}
If $j$ and $h$ are paired the cylinders defining $A_j^{(i)}$ and $A_h^{(i)}$ differ in at most one unbounded scale. 
In that index one of them satisfies the first hypothesis of either \ref{lem:pos red S left} or \ref{lem:neg red S left} and the other satisfies the first hypothesis of either \ref{lem:pos red S right} or \ref{lem:neg red S right} respectively with $\max\{|n_j|,|n_h|\}$ in place of $r$ therein.

\end{enumerate}
\end{proposition}

\begin{proof}
Let $\tau$, $m$ and $I$ be as in the statement.
The proof is recursive.

\textbf{Initial step:}
First suppose both $\scl(\lft(m))$ and $\scl(\rht(m))$ are unbounded scales.
By \cref{lem: reduction not both small} neither can be at scale $\scl(m)$ and at least one is at scale $\lambda(\scl(m))$.
If $m$ is positive let $A^{(1)}_1$ and $A^{(1)}_2$ be the sets of points in $X$ satisfying the hypothesis $9 \le x(\scl(m)) \le d(b)-\lst(r)$ of \ref{cor:pos red T left} and the hypothesis $d(b) + 2 \le x(\scl(m)) \le c(b) - \lst(r)$ of \ref{cor:pos red T right} respectively.
When $m$ is negative use corresponding hypotheses of \ref{cor:neg red T left} and \ref{cor:neg red T right} respectively instead.
In either case define $A^{(1)}_0$ to be the intersection with $X$ of the cylinders defined by the $\scl(m)$ coordinate that contribute neither to $A^{(1)}_1$ nor to $A^{(1)}_2$.
Thus $A^{(1)}_0,A^{(1)}_1,A^{(1)}_2$ is a measurable partition of $X$ and \ref{cond:partition} holds.
Set $k(1) = 2$ and $n^{(1)}_1 = \lft(m)$ and $n^{(1)}_2 = \rht(m)$.
Thus \ref{cond:setup} holds.
We get \ref{cond:agree} from the conclusions of \ref{cor:pos red T left} and \ref{cor:pos red T right} respectively.
(When $m$ is negative use \cref{cor:neg reduc conseq} instead.)
As $k(1) = 2$ we are forced to pair $n^{(1)}_1$ and $n^{(1)}_2$.
By \cref{lem: reduction not both small} at least one of $n^{(1)}_1$ and $n^{(2)}_2$ is at scale $\lambda(\scl(m))$ giving \ref{cond:differ}.
As $|\lft(m) - \rht(m)| = |\lst(m)| \q(\scl(m))$ and $2|\lst(m)| \le c(\scl(m)) + 1$ we have \ref{cond:paired numbers}.
Property \ref{cond:paired cyl} holds as the cylinders defining $A^{(1)}_1$ and $A^{(1)}_2$ are defined in the $\scl(m)$ coordinate alone to satisfy the appropriate hypotheses. 

Secondly, suppose instead that at least one of $\lft(m)$ and $\rht(m)$ is at a bounded scale greater than or equal to $\tau$.
In this case set $k(1) = 0$ and $A^{(1)}_0 = X$.
All properties except \ref{cond:Azero friends} are immediate or vacuous.
We now treat \ref{cond:Azero friends}, assuming without loss of generality that $\scl(\lft(m))=k>\tau$, a bounded scale. By Theorem \ref{thm:bounded wfriends} we get sufficiently many $(\lft(m),\scl(\lft(m)))$-friends. By \ref{cor:pos red T left} if $m>0$ or \ref{cor:neg red T left} if $m<0$, Lemmas \ref{lem:friend to weak} and \ref{lem:buddies through reduc} we have \ref{cond:Azero friends}.

Finally, suppose one of $\lft(m)$, $\rht(m)$ is at an unbounded scale and the other is at a bounded scale less than $\tau$.
(We will not find both $\lft(m)$ and $\rht(m)$ at bounded scales less than $\tau$ because $\lft(m) - \rht(m) = \lst(m) \q(\scl(m))$ forces at least one of $\lft(m)$ and $\rht(m)$ to be at a scale at least as large as $\lambda(\scl(m)) - 1$.)
Then we proceed as in the first case, except that for \ref{cond:differ} we instead note here that if $\lft(m)$ is not at scale $\lambda(\scl(m))$ then $\scl(\lft(m)) \le \lambda^2(\scl(m))$ and so $\scl(\rht(m)) = \lambda(\scl(m))$ because $\rht(m) = \lft(m) + \lst(m) \q(\scl(m))$ gives $\scl(\rht(m)) \ge \lambda(\scl(m))$ and it cannot be larger than $\lambda(\scl(m))$.

\textbf{The recursion}:
For the inductive step, suppose for some $1 \le i < I$ that a tuple
\[
( A^{(i)}_0, A^{(i)}_1, n^{(i)}_1,\dots, A^{(i)}_{k(i)}, n^{(i)}_{k(i)})
\]
satisfying all the properties has been produced.
If $k(i) = 0$ we set $k(i+1) = 0$ and $A^{(i+1)}_0 = A^{(i)}_0$.
All properties continue to hold.

Assuming, then, that $k(i) > 0$, we proceed by reducing further each $n^{(i)}_j$ unless one of the following holds.
Say that the pair $h \ne j$ in $\{1,\dots,k(i) \}$ is \define{bad} if at least one of the following is true.
\begin{itemize}
\item
$\scl(n_j^{(i)}) = \lambda^i(\scl(m))$ and $n_j^{(i)} < \frac{3}{4} \p(\lambda^i(\scl(m)))$
\item
$\scl(n_j^{(i)}) = \lambda^i(\scl(m))$ and $\lft(n_j^{(i)})$ is at a bounded scale not smaller than $\tau$
\item
$\scl(n_j^{(i)}) = \lambda^i(\scl(m))$ and $\rht(n_j^{(i)})$ is at a bounded scale not smaller than $\tau$
\item
$\scl(n_h^{(i)}) = \lambda^i(\scl(m))$ and $n_h^{(i)}<\frac{3}{4} \p(\lambda^i(\scl(m)))$
\item
$\scl(n_h^{(i)}) = \lambda^i(\scl(m))$ and $\lft(n_h^{(i)})$ is at a bounded scale not smaller than $\tau$
\item
$\scl(n_h^{(i)}) = \lambda^i(\scl(m))$ and $\rht(n_h^{(i)})$ is at a bounded scale not smaller than $\tau$
\end{itemize}
When the pair to which $1 \le j \le k(i)$ belongs is not bad we call $j$ \define{reducible} if $\scl(n_j^{(i)}) = \lambda^{i}(\scl(m))$.
Note that \ref{cond:differ} says that if neither element of a pair is bad then at least one element of each pair is reducible.
As a reducible $j$ does not belong to a bad pair it must be the case that $\lft(n^{(i)}_j)$ and $\rht(n^{(i)}_j)$ are at unbounded scales.
For each $1 \le j \le k(i)$ that is reducible set
\[
F(j)
=
\begin{cases}
\{ x \in A^{(i)}_j : x(\scl(n^{(i)}_j)) \in [0,8] \cup [d-e+1, d+1] \cup [c-e+1, c] \}
&
\text{if }
n^{(i)}_j>0
\\
\{ x \in A^{(i)}_j : x(\scl(n^{(i)}_j)) \in [0,9-e] \cup [d,d+1-e] \cup [c-1,c] \}
&
\text{if }
n^{(i)}_j < 0
\end{cases}
\]
where $c = c(\lambda^i(\scl(m)))$ and $d = d(\lambda^i(\scl(m)))$ and $e = \lst(n^{(i)}_j)$.
For each $1 \le h \le k(i)$ that is neither in a bad pair nor reducible set
\[
F(h)
=
\begin{cases}
\{ x \in A^{(i)}_h : x(\scl(n^{(i)}_j)) \in [0,8] \cup [d-e+1, d+1] \cup [c-e+1, c] \}
&
\text{if }
n^{(i)}_k > 0
\\
\{ x \in A^{(i)}_h : x(\scl(n^{(i)}_j)) \in  [0,9-e] \cup [d,d+1-e] \cup [c-1,c] \}
&
\text{if }
n^{(i)}_k < 0
\end{cases}
\]
where $c = c(\lambda^i(\scl(m)))$ and $d = d(\lambda^i(\scl(m)))$ and $e = \lst(n^{(i)}_k)$ and $j$ is the (necessarily reducible) pair of $h$ as in \ref{cond:paired cyl}.

Note all $1 \le h \le k(i)$ that are neither bad nor reducible are paired with an index that is reducible and not bad. We define
\begin{equation}
\label{eqn:A0 recursive def}
A_0^{(i+1)}
=
A_0^{(i)}
\cup
\bigcup_{j \text{ in a bad pair}} A_j^{(i)}
\cup
\bigcup_{j \text{ reducible }} F(j)
\cup
\bigcup_{h \text{ neither bad nor reducible}} F(h)
\end{equation}
and note that the sets making up \eqref{eqn:A0 recursive def} are pairwise disjoint as $A^{(i)}_0,\dots,A^{(i)}_{k(i)}$ is a partition of $X$.
We now verify \ref{cond:Azero friends} for $A^{(i+1)}_0$ by checking \ref{cond:Azero friends} holds for each of its constituents.
\begin{itemize}
\item
It follow by the inductive assumption for $A_0^{(i)}$.
\item
Fix a bad pair $h \ne j$.
We first check that
\begin{equation}
\label{eqn:bad pair sets comparable}
\dfrac{1}{2} < \dfrac{\mu(A^{(i)}_j)}{\mu(A^{(i)}_k)} < 2
\end{equation}
holds.
By \ref{cond:partition} and \ref{cond:paired cyl} there exist 3 unions of cylinders, $B,C,D$ all defined only by fixing indices at unbounded scales greater than $\tau$ so that $A_j^{(i)}=B \cap C$ and $A_k^{(i)}=B \cap D$.
What is more, by \ref{cond:paired cyl} and our construction, up to flipping the role of $C$ and $D$, $C$ is defined by all of the left towers at one unbounded index, say $a(n)$ and $D$ is defined by all of the right towers $a(n)$. Moreover, this index is smaller than the indices defining $B$.
Now by a similar computation to those we have done before
\[
1 < \frac{(d(n)+1)\P(n)-\nu(X)^{-1}\nu(\bigcup_{\ell>n}W_{a(\ell)})}{(d(n)+1)(\P(n)-\Q(n))}\leq \frac{\mu(C)}{\mu(D)}<1+2\frac{\Q(n)}{\P(n)}
\]
and \eqref{eqn:bad pair sets comparable} follows from \cref{lem:indep} and the fact that the index defining $C$ and $D$ is at least $2^{2^5}$ smaller than the smallest index defining $B$.

With \eqref{eqn:bad pair sets comparable} established it suffices to find a positive proportion of $(m,\tau)$ buddies in either $A^{(i)}_j$ or $A^{(i)}_k$.
How this is done depends on why the pair $h,j$ is bad.
\begin{itemize}
\item 
If $\scl(n^{(i)}_j) = \lambda^i(\scl(m))$ and $s \in \{ \lft(n^{(i)}_j), \rht(n^{(i)}_j) \}$ at a bounded scale not smaller than $\tau$ apply \cref{thm:bounded wfriends} with $s$ in place of $r$ and note that the set $B$ therein is defined by coordinates no larger than $\scl(s) + 1$ while $A^{(i)}_j$ is defined by coordinates no smaller than $\lambda^{i-1}(\scl(m))$.
We may therefore apply \cref{lem:indep} to get
\[
\mu(B \cap A^{(i)}_j) \ge \left( 1 - \dfrac{1}{2^{2^5}} \right) \mu(B) \mu(A^{(i)}_j)
\]
whence a positive proportion of $A^{(i)}_j$ intersects $B$.
Removing $\bad(\scl(s),8s)$ and \eqref{eqn:bad ancestors} leaves us with a positive proportion of $(m,\tau)$ buddies, as in the proof of \cref{cor:bounded mmix}.
\item
If $\scl(n^{(i)}_j) = \lambda^i(\scl(m))$ and $n^{(i)}_j < \tfrac{3}{4} \p(\lambda^i(\scl(m)))$ apply \cref{prop:small friend}.
\item
If instead $\scl(n^{(i)}_h) = \lambda^i(\scl(m))$ the arguments are exactly the same upon replacing $j$ with $h$.
\end{itemize}
\item
When $j$ is reducible or neither bad nor reducible, we must show that $F(j)$ has a positive proportion of $(m,j)$-buddies for $j\geq \tau$.
\begin{itemize}
\item
We claim that it suffices to show this when $j$ is reducible, because the measure of the paired cylinders differs by less than a factor of 2 and when $h$ is neither bad nor reducible it is paired with $j$ which is reducible. Indeed, by \ref{cond:paired cyl} they differ by at most one unbounded symbol. We consider the other (necessarily unbounded) entries defining these sets \ref{cond:partition} and apply Lemma \ref{lem:indep} to these cylinders and the cylinders defined by the (necessarily unbounded) position where they are defined to differ. 
\item
By Lemma \ref{lem:friends}, if $j$ is reducible the set of $(n_j^{(i)},\lambda^{i}(\scl(m)))$-friends is at least $\frac{1}{10^6}\mu(F(j))$.
\item
By Lemmas \ref{lem:friend to weak} and \ref{lem:buddies through reduc} pairs that are $(n_j^{(i)},\lambda^{i}(\scl(m)))$-friends and for which \eqref{eq:same coords} hold for both elements of the pair are automatically $(m,\tau)$-buddies.
\end{itemize}
\end{itemize}

Now we define $k(i+1)$ and the sets $A^{(i+1)}_1,\dots,A^{(i+1)}_{k(i+1)}$ and the times $n^{(i+1)}_1,\dots,n^{(i+1)}_{k(i+1)}$.

Write
\[
G =  \{ 1 \le j \le k(i) : j \textup{ reducible} \} \cup  \{ 1 \le h \le k(i) : h \textup{ neither bad nor reducible} \}
\]
which consists of those $1 \le j \le k(i)$ belonging to a pair that is not bad.
Temporarily write $c = c(\lambda^i(\scl(m)))$ and $d = d(\lambda^i(\scl(m)))$.
There are two cases to consider.
\begin{itemize}
\item
When $j \in G$ is reducible and its pair is too, write $e = \lst(n^{(i)}_j)$ and define the sets
\begin{align*}
L^{(i)}_j &= \{ x \in A^{(i)}_j : x(\scl(n^{(i)}_j)) \in [9,d-e] \} \\
R^{(i)}_j &= \{ x \in A^{(i)}_j : x(\scl(n^{(i)}_j)) \in [d+2,c-e] \}
\end{align*}
if $n^{(i)}_j > 0$ and
\begin{align*}
L^{(i)}_j &= \{ x \in A^{(i)}_j : x(\scl(n^{(i)}_j)) \in [10-e,d-1] \} \\
R^{(i)}_j &= \{ x \in A^{(i)}_j : x(\scl(n^{(i)}_j)) \in [d+2-e,c-2] \}
\end{align*}
if $n^{(i)}_j < 0$.
In this case form the pairs $(L^{(i)}_j, \lft(n^{(i)}_j))$ and $(R^{(i)}_j, \rht(n^{(i)}_j))$.
\item
For neither bad nor reducible $h \in G$ its paired index $j \in G$ is necessarily reducible by \ref{cond:differ}. Again write $e = \lst(n^{(i)}_j)$ and in this case define
\begin{align*}
L^{(i)}_h &= \{ x \in A^{(i)}_h : x(\scl(n^{(i)}_j)) \in [9,d-e] \} \\
R^{(i)}_h &= \{ x \in A^{(i)}_h : x(\scl(n^{(i)}_j)) \in [d+2,c-e] \}
\end{align*}
if $n^{(i)}_j > 0$ and
\begin{align*}
L^{(i)}_h &= \{ x \in A^{(i)}_h : x(\scl(n^{(i)}_j)) \in [10-e,d-1] \} \\
R^{(i)}_h &= \{ x \in A^{(i)}_h : x(\scl(n^{(i)}_j)) \in [d+2-e,c-2] \}
\end{align*}
if $n^{(i)}_j < 0$.
In this case form the pairs $(L^{(i)}_h, n^{(i)}_h)$ and $(R^{(i)}_h, n^{(i)}_h)$.
\end{itemize}

Take $k(i+1) = 2|G|$.
Let $(A^{(i+1)}_1, n^{(i+1)}_1),\dots,(A^{(i+1)}_{k(i+1)}, n^{(i+1)}_{k(i+1)})$ be an enumeration of the $k(i+1)$ pairs that result from the above two cases.
For each $1 \le j \le k(i+1)$ let $1 \le \alpha(j) \le k(i)$ be the index for which $n^{(i)}_{\alpha(j)}$ determines $n^{(i+1)}_j$ as the above pairs are formed.
For each $g \in G$ the sets $F(g)$, $L^{(i)}_g$, $R^{(i)}_g$ partition $A^{(i)}_g$.
It is immediate from the recursion and the definitions that \ref{cond:setup} and \ref{cond:partition} are satisfied. We have already verified \ref{cond:Azero friends}.

We next verify \ref{cond:agree}.
Fix $1 \le j \le k(i+1)$.
If its ancestor $1 \le \alpha(j) \le k(i)$ is such that $n^{(i+1)}_j = n^{(i)}_j$ then there is nothing to check.
Suppose instead that its ancestor is a value $1 \le \alpha(j) \le k(i)$ for which $n^{(i+1)}_j = \lft(n^{(i)}_{\alpha(j)})$ or $n^{(i+1)}_j = \rht(n^{(i)}_{\alpha(j)})$.
Fix $x \in A^{(i+1)}_j$ with $j \ne 0$.
Write $n = n_{\alpha(j)}^{(i)}$ and let $r \in \{\lft(n),\rht(n)\}$ be the reduction we are considering. By either \cref{cor:pos reduc conseq} or \cref{cor:neg reduc conseq} we have that $(T^n x)(\ell) = (T^r  x)(\ell)$ for $\ell \leq\scl(n)-1$, so long as $x \notin \bad(\scl(n),n)$. 
Invoking \ref{cond:agree} gives $(T^n x)(\ell) = (T^m x)(\ell)$ for $\ell\leq \lambda^{i-1}\scl(m)-1$ so long as \eqref{eqn:bad ancestors} holds. Combining the two verifies \ref{cond:agree}.

It remains to define how $1,\dots,k(i+1)$ are paired and verify \ref{cond:differ}, \ref{cond:paired numbers}, \ref{cond:paired cyl}.
There are two cases to consider.
\begin{itemize}
\item
When $j,h \in G$ are paired and both are reducible, pair the descendants of $j$ and pair the descendants of $h$.
\item 
When $j,h \in G$ are paired and neither is bad and only $j$ is reducible then necessarily $h$ is neither bad nor reducible. Thus $\scl(n^{(i)}_h) < \lambda^i(\scl(m))$ and since $h$ is not bad, $J(n^{(i)}_h)\leq \lambda^{i+1}(\scl(m))$. There are two sub-cases to consider.
\begin{enumerate} 
\item
If $\scl(n_h^{(i)}) = \lambda^{i+1}(\scl(m))$ then pair the descendants of $j$ and pair the descendants of $h$.
\item
If $\scl(n_h^{(i)}) < \lambda^{i+1}(\scl(m))$ we form the pairing so that $L^{(i)}_j$ is with $L^{(i)}_h$ and $R^{(i)}_j$ is with $R^{(i)}_h$.
\end{enumerate}
\end{itemize}

First we verify \ref{cond:paired cyl}. When the descendants $h'$, $h''$ of some $1 \le h \le k(i)$ are paired we have
\[
\{ A^{(i+1)}_{h'}, A^{(i+1)}_{h''} \} = \{ L^{(i)}_h, R^{(i)}_h \}
\]
and therefore the sets $A^{(i+1)}_{h'}$, $A^{(i+1)}_{h''}$ disagree only in the $\scl(n^{(i)}_{j'})$ coordinate, where $j' = h$ when $h$ is reducible and $j'$ is the pair of $h$ when $h$ is not reducible. In either case $\scl(n^{(i)}_{j'})$ is unbounded.
It remains to verify that the pairing satisfies \ref{cond:paired cyl} when we are in the second subcase. In that case the additional restriction defining $L^{(i)}_j$ and $L^{(i)}_h$ are the same because both are determined by $n^{(i)}_j$ so $L^{(i)}_j$ and $L^{(i)}_h$ will only disagree in the coordinate where $A^{(i)}_j$ and $A^{(i)}_h$ disagreed. The same is true for $R^{(i)}_j$ and $R^{(i)}_h$.

To conclude we check \ref{cond:differ} and \ref{cond:paired numbers} simultaneously.
Fix again $1 \le j,h \le k(i)$ paired.
The construction is such that each descendent is paired with exactly one other descendant.
Now we check at least one descendant in each pair has its associated time at scale $\lambda^{i+1}(m)$.
If both $n_j^{(i)}$ and $n_h^{(i)}$ are reducible we can apply \cref{lem: reduction not both small} to both separately. Note that in this case \eqref{eq:differnce in reduction} gives \ref{cond:paired numbers}. This is also the justification of Subcase 1. 
We are left with Subcase 2 where $n_j^{(i)}$ is reducible and $n_h^{(i)}$ is neither bad nor reducible.
We are done if $\scl(n_h^{(i)}) = \lambda^{i+1}(\scl(m))$, so we assume $\scl(n_h^{(i)}) \le \lambda^{i+2}(\scl(m))$.
By \ref{cond:paired numbers}
\[
n_j^{(i)}=k \q(\lambda^{i-1}\scl(m))+n_h^{(i)} \text{ with }0<|k|\leq 2d(\lambda^{i})+3
\]
and -- applying \cref{lem:for pairing} with $b = n^{(i)}_h$ and using that $n^{(i)}_j$ is not bad -- we have that both $\lft(n_j^{(i)})$ and $\rht(n_j^{(i)})$ have the form
\[
k'_L \q(\lambda^i \scl(m)) + n_h^{(i)} \text{ and } k'_R \q(\lambda^i \scl(m)) + n_h^{(i)}
\]
with $0<|k'_L|,|k'_R|\leq c(\lambda^i(J(m))) + 1$. 
So both pairs created from $n_h^{(i)}$ and $n_j^{(i)}$ satisfy \ref{cond:differ}, with the terms $\lft(n_j^{(i)})$ and $\rht(n_j^{(i)})$ satisfying 
\[
\scl(\lft(n_j^{(i)})) = \lambda^{i+1}(\scl(m)) = \scl(\rht(n_j^{(i)}))
\]
as desired.
Similarly, we have both pairs satisfy \ref{cond:paired numbers}.
\end{proof}

\begin{proposition}
\label{prop:reduction multiples}
Fix $\tau=a(\beta) \ge a(10)$ a target unbounded scale and $m \in \N$ with $\scl(m)=a(b)$ unbounded and $b>\beta$.
Let $I \in \N$ and the tuples $(A^{(i)}_0, A^{(i)}_1, n^{(i)}_1,\dots, A^{(i)}_{k(i)}, n^{(i)}_{k(i)})$ for $1 \le i \le I=b-\beta$ result from an application of \cref{prop:reduction} to $m$.
In particular, $\scl(n_j^{(i)}) \leq \tau$ and for at least half of the $j$ we have $\scl(n_j^{(I)})=\tau$.
If $\mu(A^{(I)}_0)<\frac 1 {10^{50}}$  and tuples $(\overline{A}^{(i)}_0, \overline{A}^{(i)}_1, \overline{n}^{(i)}_1,\dots, \overline{A}^{(i)}_{\overline{k}(i)}, \overline{n}^{(i)}_{\overline{k}(i)})$ for $1 \le i \le K = b-\beta$ result from applying \cref{prop:reduction} to $6d(\beta) m$ then $\mu(\overline{A}^{(K)}_0)>\frac 1 {10^{50}}$.
\end{proposition}
This follows from repeated applications of the following result.

\begin{lemma}
\label{lem:reduc multiply}
Fix $m \in \Z$ with $\scl(m) = a(b)$ an unbounded scale.

If $\rht(m)$ is at an unbounded scale and $k \in \Z$ satisfies $|k|\leq 6d(b)$ and $km$ is at an unbounded scale then one of the following is true.
\begin{itemize} 
\item
$\rht(km)=k\rht(m)$
\item 
{$\rht(km)$ is at a bounded scale} 
\item
At least one of $\rht(m)$, $\rht(km)$ is both at scale $\scl(m)$ and less than $\frac{3}{4} \p(\scl(m))$
\end{itemize}
Moreover, if $\rht(km)$ is at scale $\lambda(\scl(m))$ then $|k| < 6d(a^{-1}(\scl(\rht(m))))$.

If $\lft(m)$ is at an unbounded scale and $k \in \Z$ satisfies $|k|\leq 6d(b)$ and $km$ is at an unbounded scale then one of the following is true.
\begin{itemize} 
\item
$\lft(km) = k\lft(m)$
\item 
{$\lft(km)$ is at a bounded scale }

\end{itemize}
Moreover, if $\lft(km)$ is at scale $\lambda(\scl(m))$ then $|k| < 6d(a^{-1}(\scl(\lft(m)))$.
\end{lemma}
\begin{proof}[Proof of Proposition \ref{prop:reduction multiples} assuming Lemma \ref{lem:reduc multiply}]
Consider any sequence $m, n^{(1)},n^{(2)},\dots,n^{(I)}$ of reductions of $m$.
The corresponding sequence for $km$ is $km,kn^{(1)},kn^{(2)},\dots$ until one of $kn^{(i)}$ is in a bad pair.
This follows by \cref{lem:reduc multiply} and the definition of bad pairs in the proof of \cref{prop:reduction}.
Now if $kn^{(i)}$ is bad the corresponding points are put in $A_0$ and we have proportional friends.
On the other hand $6d(\beta) n^{(I)} > \p(\tau)$ so $A^{(I-1)}$ is put in $A_0^{(I)}$ if it is not put in earlier.
\end{proof}

\begin{proof}[Proof of Lemma \ref{lem:reduc multiply}]
We present the proof in the case of $\rht$. The case of $\lft$ is similar.
First, if $km$ is at an unbounded scale, then $\scl(km) = \scl(m)$ because the bound on $k$ gives $\scl(km) \le \scl(m) + 1$.
If $\rht(km) \neq k\rht(m)$ then
\begin{align*}
& \Big\vert\rht(km)-k\rht(m)\Big\vert \\
= & \left\vert km- \lst(km) \Big(\p (\scl (m))-\q( \scl(m))\Big)-k\Big( m - \lst(m) \big(\p(\scl(m))-q(\scl(m))\big)\Big) \right\vert \\
\ge & \p(\scl(m))-\q(\scl(m))
\end{align*}
and by \cref{lem:reduction LR} they are both at most $\frac 3 4 \p(J(m))$.
Lastly, if $\scl(m)=a(\gamma))$ and $6d(\gamma+1)\geq|k|\geq 6d(\gamma)$ then $\scl(km)$ is strictly between $a(\gamma)$ and $a(\gamma+1)$. Thus it is at a bounded scale.
\end{proof}

\subsection{Proof of \cref{thm:mild mixing}}
\label{subsec:mild mixing}

We are ready to conclude that $(X,\mu,T)$ is mild mixing.

\begin{proof}[Proof of \cref{thm:mild mixing}]
 
Fix a non-constant function $f$ in $\ltwo(X,\mu)$ and a sequence $\textit{\textbf{r}}$ in $\Z$ with $|r(n)| \to \infty$.
Set $\rho=10^{-50}\cdot 10^{-9}$.
Let $\delta$ and $\tau'$ be as in \cref{cor:end} for our choices of $f$ and $\rho$.
Fix $\tau = a(\beta) > \tau'$ an unbounded scale.
It suffices to show that at least one of $\textbf{\textit{r}}$ or $6 d(\beta) \textbf{\textit{r}}$ is not a rigidity sequence for $f$.
Indeed if $\textit{\textbf{r}}$ is a rigidity sequence for $f$ then so is $k \textit{\textbf{r}}$ for any fixed $k \in \mathbb{Z}\setminus \{0\}$.
By \cref{thm:bounded wfriends} and by passing to a subsequence if necessary  we may first assume for all $n \in \N$ that $\scl(r(n))$ and $\scl(6d(\beta) r(n))$ are unbounded scales  and that $\scl(r(n)) \ge \tau$.

Fix $I(n)$ and $K(n)$ maximal with $\lambda^{I(n)}(r(n)) \ge \tau$ and $\lambda^{K(n)}(6d(\beta) r(n)) \ge \tau$ respectively.
Apply \cref{prop:reduction} to $r(n)$ and $6 d(\beta) r(n)$.
Let $\Gamma_n \subset A^{(I(n))}_0$ and $\overline{\Gamma}_n \subset \overline{A}^{(K(n))}_0$ be the respective sets furnished by \ref{cond:Azero friends}.
There are measure preserving maps $\phi_n : \Gamma_n \to X$ and $\overline{\phi}_n : \overline{\Gamma}_n \to X$ with the following properties.
\begin{itemize}
\item
For each $z \in \Gamma_n$ we have that $(z,\phi_n(z))$ are $(r(n),\tau)$ buddies.
\item
For each $z \in \overline{\Gamma}_n$ we have that $(z,\overline{\phi}_n(z))$ are $(6d(\beta)r(n),\tau)$ buddies.
\item
$\mu(\Gamma_n) \ge 10^{-9}\mu(A^{(I(n))}_0)$

\item
$\mu(\overline{\Gamma}_n) \ge 10^{-9}\mu(\overline{A}^{(K(n))}_0)$

\end{itemize} 
By \cref{prop:reduction multiples} 
\[
\max\{\mu(A^{(I(n))}_0), \mu(\overline{A}^{(K(n))}_0) \} \geq 10^{-50}
\]
and thus \cref{cor:end} gives
\[
\max \left\{ \int |f-f\circ T^{r(n)}| \intd \mu,\int |f-f \circ T^{6d(\beta)r(n)}| \intd\mu \right\} > \delta
\]
proving the theorem.

\end{proof}

\appendix
\renewcommand{\thesection}{Appendix \Alph{section}}

\section{Changes for the proof of Proposition \ref{prop:ce}}
\label{app:ce_proof}

In this section we prove the version of \cite[Corollary 3.3]{MR4269425} that is needed to prove \cref{prop:ce}.

The next lemma and its proof are unchanged from \cite[Lemma 3.2]{MR4269425}.
\begin{lemma}
\label{lemma:to:bary}
Given $c>0$ and $d \in \mathbb{N}$ there exists $\rho<1$, $C$ so that
if $0 < \delta_i <1/2$ and $a_i, b_i$ are such that $a_i,b_i>c$ and $1\geq
a_i+b_i>1-\delta_i$ and also $0 \le \zeta_i^{(\ell)}\le 1$ are
sequences of real numbers for each $\ell \in \{1,...,d\}$ satisfying
\begin{equation}
\label{eq:gammas:assumption}
|\zeta_i^{(\ell)}-(a_i\zeta_{i-1}^{(\ell-1)}+b_i\zeta_{i-1}^{(\ell)})|<\delta_{i-1}
\end{equation}
then 
$$\left|\zeta_i^{(s)}-\frac 1 d \sum_{\ell=1}^d\zeta_k^{(\ell)}\right|\leq
C\sum_{j=k}^{i-1}\left(\delta_i+\frac{\delta_i}{1-\delta_i}\right)+C\rho^{i-k}$$ for
all $k \ge 0$, $i>k$ and $s \in \{1,...,d\}$. 
\end{lemma}

\begin{corollary}
\label{cor:to:bary}
Under the assumptions of Proposition~\ref{prop:ce} there
exist $ \rho<1$, $C'>0$ so that $d_{KR}(\nu_k^{(\ell)},\frac 1
d\sum_{\ell=1}^d \nu_{b}^{(\ell)})\leq C'\sum_{j=b}^k\epsilon_j
+C'\rho^{k-b}$ whenever $k\geq b$ and $\ell\in \{1,\dots,d\}$.
\end{corollary}

\begin{proof}[Proof of Corollary~\ref{cor:to:bary}]
First notice that by \ref{ce:nearness}  we have that 
\begin{equation}\label{eq:bary1}
d_{KR}(\nu_{\alpha(j)}|_{A(j)},\nu_{\gamma(j-1)}|_{A(j)})<\epsilon(j)\text{ and }d_{KR}(\nu_{(\alpha(j)}|_{B(j)},\nu_{\alpha(j-1)}|_{B(j)})<\epsilon(j).
\end{equation}
We now claim that there exists a constant $D$ so that for all $j$, 
\begin{equation}
\label{eq:KR est2}
d_{KR}\left(\frac 1 {\mu(A(j))} \nu_{\gamma(j-1)}|_{A(j)},\nu_{\gamma(j-1)}\right)<D(\epsilon(j-1)+\epsilon(j))
\end{equation}
By definition of $A(j)$ we have that  
$$\int f d\nu(\gamma(j-1))|_{A(j)}=\int_{\cup_{i=0}^{r(j)-1}T^iJ(j)\setminus U(j)}f(x,T^{\gamma(j-1)}x).$$

 Now, 
\begin{multline}\label{eq:window}
\bigg|\int_{\cup_{i=0}^{r(j)-1} T^iJ(j)}f\circ (T \times T)^i(x,T^{\gamma(j-1)}x)-\sum_{j=0}^{(1-\epsilon(j-1))r(j)-1}\frac 1 {\epsilon(j-1)r(j)}\int_{\cup_{i=0}^{\epsilon(j-1)r(j)}T^{i+j} J(j)}f\circ (T \times T)^i(x,T^{\gamma(j-1)}x)\bigg| \\ \leq 2(\epsilon(j-1)r(j))\mu(J(j))\|f\|_{\sup}.
\end{multline}

To see \eqref{eq:window}, the integral has the form
$$\int_{\bigcup_{i=0}^{r(j)}T^iJ(j)}f(x,T^{\gamma(j-1)}x)-c(x)f(x,T^{\gamma(j-1)}x)$$ where $0\leq c(x)\leq 1$ for all $x\in \bigcup_{i=0}^{r(j)-1}T^iJ(j)$ and $c(x)=1$ for all $x \in \bigcup_{i=e(j-1)r(j)}^{(1-e(j-1))r(j)-1}T^i J(j)$, giving \eqref{eq:window}.

By \ref{ce:kr},

\begin{multline}\label{eq:CEchanged}
\bigg| \frac 1 {(1-\epsilon(j-1))r(j)}\sum_{j=0}^{(1-\epsilon(j-1))r(j)}\frac{1}{\mu(J(j))\epsilon(j-1)r(j)}
\\ \int_{\bigcup_{i=0}^{\epsilon(j-1)r(j)-1}T^iT^kJ(j)}f \circ (T \times T)^i(x,T^{\gamma(j-1)}x)
d\mu -  \int_{X \times X} f \,
d\nu_{\gamma(j-1)}\bigg|  
\le 3 \|f\|_{\text{Lip}} \epsilon({j-1})+4\|f\|_{\sup}\epsilon(j).
\end{multline}
To see this, we split the left hand side into two parts: the $x\in \bigcup_{k=0}^{(1-\epsilon(j-1)r(j)}T^kJ_j$ for which \eqref{eq:ce changed} holds and the complement that we estimate trivially.

Because $A(j)=\bigcup_{j=0}^{r(j)-1}T^jJ(j)\setminus U(j)$
we now use  \ref{ce:discard} (i.e.\ the size estimate on $U_j$),

\begin{multline}
   \bigg| \frac 1 {(1-\epsilon(j-1))r(j)}\sum_{j=0}^{(1-\epsilon(j-1))r(j)}\frac{1}{\lambda(J(j))\epsilon(j-1)r(j)}\\
 \int_{\bigcup_{i=0}^{\epsilon(j-1)r(j)-1}T^iT^kJ(j)}f \circ (T \times T)^i(x,T^{\gamma(j-1)}x)\chi_{U(j)}(x)
d\mu\bigg| \\ 
\le \|f\|_{\sup} \mu(U(j))
\le \|f\|_{\sup} \epsilon(j).
\end{multline}

Combining these estimates we have \eqref{eq:KR est2}.

Similarly, by partitioning $B_j$ into $D_{\frac{r_j}2},...$ where 
$$D_\ell=\{x\in S^{r_j}J_j: \min\{i>0:S^ix\in J_j\}=\ell\},$$ 
we get that there exists $E>0$ so that 
\begin{equation}
\label{eq:KR:est:2prime}
d_{KR}\left(\frac 1 {\mu(B(j)} \nu_{(\alpha(j-1))}|_{B(j)},\nu_{(\alpha(j-1))}\right)<E(\epsilon(j-1)+\epsilon(j)).
\end{equation}

We have analogous versions of these estimates swapping the roles of $\alpha(j)$ for $\gamma(j)$, $\alpha(j-1)$ for $\gamma(j-1)$ and $\gamma(j-1)$ for $\alpha(j-1)$.

 So for any 1-Lipschitz function, $f$, with $\|f\|_{\sup}\leq 1$, we
 claim that we may apply Lemma~\ref{lemma:to:bary} to 
$\zeta_i^{(1)}=\int fd\nu_{\alpha(i)}$, $\zeta_i^{(2)}=\int f d\nu_{\gamma(i)}$ with $c=c$,
$\delta_{j-1}=(E+D)(\epsilon(j-1)+\epsilon(j))$,
$a_j=\mu(A(j))$ and $b_j=\mu(B(j))$. To verify
(\ref{eq:gammas:assumption}), note that
$$\left|\int f
d\nu_{\alpha(j)}-\int_{A(j)}fd\nu_{\alpha(j)}-\int_{B(j)}fd\nu_{\alpha(j)}\right|\leq
\|f\|_{\sup}\lambda(U(j))<\epsilon(j)$$ and similarly for $\gamma(j)$. Thus,  by (\ref{eq:bary1})
$$\left|\int fd \nu_{\alpha(j)}-\int_{A(j)} f\, d\nu_{\gamma(j-1)}-\int_{B(j)} f
d\nu_{\alpha(j-1)}\right|\leq 2\epsilon(j).$$
Then, by (\ref{eq:KR est2}) and (\ref{eq:KR:est:2prime}),
$$\left|\int fd \nu_{\alpha(j)}-\Big(\mu(A(j))\int
fd\nu_{\gamma(j-1)}+\mu(B(j))\int f
d\nu_{\alpha(j)}\Big)\right|\leq (E+D)(\epsilon(j-1)+\epsilon(j)).$$
This and the analogous result for $\nu_{\gamma(j)}$, completes the verification of (\ref{eq:gammas:assumption}), and,
in view of Lemma~\ref{lemma:to:bary}, the proof of
Corollary~\ref{cor:to:bary}. 
\end{proof}

\section{Notation}

The following symbols - used throughout the paper - are recalled here for reference.

\subsection*{Odometer notation}

\begin{itemize}
\item
$\Omega$ is the odometer space
\item
$n \mapsto a(n)$ enumerates those sparse positions where the odometer has many digits
\item
$n \mapsto 2d(n)+2$ is the number of digits in the odometer at position $a(n)$
\item
$n \mapsto c(n) + 1$ is the number of digits in the $n$th position of the odometer
\item
$n \mapsto t(n)$ is the number of cylinders of length $n$ in $\Omega$
\item
$\nu$ is the normalized Haar measure on $\Omega$
\end{itemize}

\subsection*{System notation}

\begin{itemize}
\item
$n \mapsto W_n$ describes the cylinders of length $n$ removed from the odometer
\item
$Y_n = \Omega \setminus (W_1 \cup \cdots \cup W_n)$
\item
$X$ is the space $\Omega \setminus \bigcup\limits_{n \in \N} W_n$
\item
$T$ is the transformation we consider on $X$
\item
$\mu$ is the normalized restriction of $\nu$ to $X$
\item $\bad(b,r) = \displaystyle\bigcup_{a(m) > b} \{x \in X : (S|Y_b)^i(x) \in W_{a(m)} \textup{ for some } -|r| \le i \le |r|\}$
\end{itemize}

\subsection*{Specific times}

\begin{itemize}
\item
$\p(n)$ is the minimal natural number for which $(S|Y_n)^{\p(n)}[0^{n-1} 0] = [0^{n-1} 1]$
\item
$\P(n) = \p(a(n))$ is the height of the left tower at scale $a(n)$
\item
$\Q(n) = 2\mathsf{P}(n-1) - \Q(n-1)$ is the number of levels removed from $\Omega$ at scale $a(n)$
\item
$\q(n) = \Q(a^{-1}(n))$ when $n$ unbounded and zero otherwise
\end{itemize}

\subsection*{Expansion notation}

\begin{itemize}
\item
$\lst(r)$ is the leading coefficient in the tower expansion of $r \in \Z$
\item
$\scl(r)$ is the leading scale in the tower expansion of $r \in \Z$
\end{itemize}

\subsection*{Reduction notation}

\begin{itemize}
\item
$\lft(r) = r - \lst(r) \p(\scl(r))$ whenever $\scl(r)$ is unbounded
\item
$\rht(r) = r - \lst(r) \p(\scl(r)) + \lst(r) \q(\scl(r))$ whenever $\scl(r)$ is unbounded
\item
$\lambda(n)$ is the largest unbounded scale strictly smaller than $n$ 
\end{itemize}

\printbibliography

\end{document}